\documentclass[11pt,reqno]{amsart}

\usepackage{amsmath, amsthm, amssymb}
\usepackage{amsfonts}
\usepackage[ansinew]{inputenc}
\usepackage[dvips]{epsfig}
\usepackage[english]{babel}
\usepackage{mathptmx}

\usepackage{ifthen}

\usepackage{hyperref}
\hypersetup{
     colorlinks   = true,
		 linkcolor    = black,
     citecolor    = gray
}

\usepackage{cite}
\usepackage{graphicx}
\usepackage{amscd}
\usepackage[dvipsnames]{xcolor}
\usepackage{bm}
\usepackage{enumerate}

\usepackage{verbatim}
\usepackage{hyperref}
\usepackage{amstext}
\usepackage{latexsym}

\let\oldsqrt\sqrt
\def\sqrt{\mathpalette\DHLhksqrt}
\def\DHLhksqrt#1#2{%
	\setbox0=\hbox{$#1\oldsqrt{#2\,}$}\dimen0=\ht0
	\advance\dimen0-0.2\ht0
	\setbox2=\hbox{\vrule height\ht0 depth -\dimen0}%
	{\box0\lower0.4pt\box2}}

\def\vint_#1{\mathchoice%
      {\mathop{\kern 0.2em\vrule width 0.6em height 0.69678ex depth -0.58065ex
              \kern -0.8em \intop}\nolimits_{\kern -0.4em#1}}%
      {\mathop{\kern 0.1em\vrule width 0.5em height 0.69678ex depth -0.60387ex
              \kern -0.6em \intop}\nolimits_{#1}}%
      {\mathop{\kern 0.1em\vrule width 0.5em height 0.69678ex depth -0.60387ex
              \kern -0.6em \intop}\nolimits_{#1}}%
      {\mathop{\kern 0.1em\vrule width 0.5em height 0.69678ex depth -0.60387ex
              \kern -0.6em \intop}\nolimits_{#1}}}
\def\vintslides_#1{\mathchoice%
      {\mathop{\kern 0.1em\vrule width 0.5em height 0.697ex depth -0.581ex
              \kern -0.6em \intop}\nolimits_{\kern -0.4em#1}}%
      {\mathop{\kern 0.1em\vrule width 0.3em height 0.697ex depth -0.604ex
              \kern -0.4em \intop}\nolimits_{#1}}%
      {\mathop{\kern 0.1em\vrule width 0.3em height 0.697ex depth -0.604ex
              \kern -0.4em \intop}\nolimits_{#1}}%
      {\mathop{\kern 0.1em\vrule width 0.3em height 0.697ex depth -0.604ex
              \kern -0.4em \intop}\nolimits_{#1}}}

\newcommand{\aveint}[2]{\mathchoice%
      {\mathop{\kern 0.2em\vrule width 0.6em height 0.69678ex depth -0.58065ex
              \kern -0.8em \intop}\nolimits_{\kern -0.45em#1}^{#2}}%
      {\mathop{\kern 0.1em\vrule width 0.5em height 0.69678ex depth -0.60387ex
              \kern -0.6em \intop}\nolimits_{#1}^{#2}}%
      {\mathop{\kern 0.1em\vrule width 0.5em height 0.69678ex depth -0.60387ex
              \kern -0.6em \intop}\nolimits_{#1}^{#2}}%
      {\mathop{\kern 0.1em\vrule width 0.5em height 0.69678ex depth -0.60387ex
              \kern -0.6em \intop}\nolimits_{#1}^{#2}}}
    
\makeatletter
\def\moverlay{\mathpalette\mov@rlay}
\def\mov@rlay#1#2{\leavevmode\vtop{%
   \baselineskip\z@skip \lineskiplimit-\maxdimen
   \ialign{\hfil$\m@th#1##$\hfil\cr#2\crcr}}}
\newcommand{\charfusion}[3][\mathord]{
    #1{\ifx#1\mathop\vphantom{#2}\fi
        \mathpalette\mov@rlay{#2\cr#3}
      }
    \ifx#1\mathop\expandafter\displaylimits\fi}
\makeatother

\allowdisplaybreaks

\newcommand{\R}{\mathbb{R}} 
\newcommand{\N}{\mathbb{N}} 

\newcommand{\dist}{\textnormal{dist}} 
\newcommand{\diam}{\textnormal{diam}} 
\newcommand{\supp}{\textnormal{supp}} 

\DeclareMathOperator{\essinf}{essinf}

\newcommand{\eps}{\epsilon}
\renewcommand{\phi}{\varphi}

\newcommand{\cD}{{\mathcal D}}
\newcommand{\cE}{{\mathcal E}}

\newcommand{\cH}{{\mathcal H}}

\newcommand{\cL}{{\mathcal L}}

\newcommand{\cV}{{\mathcal V}}

\newcommand{\loglap}{L_{\text{\tiny $\Delta \,$}}}

\theoremstyle{definition}
\newtheorem{defi1}{Definition}[section]
\newtheorem{defi}[defi1]{Definition}

\newtheorem{rem}[defi1]{Remark}

\theoremstyle{plain} 
\newtheorem{thm}[defi1]{Theorem}

\newtheorem{prop}[defi1]{Proposition}

\newtheorem{lemma}[defi1]{Lemma}

\newtheorem{cor}[defi1]{Corollary}

\theoremstyle{definition}

\numberwithin{equation}{section} 

\title{Nonlocal operators of small order}

\author[]
{Pierre Aime Feulefack and Sven Jarohs}

\address{Goethe-Universit\"{a}t Frankfurt, Institut f\"{u}r Mathematik.
	Robert-Mayer-Str. 10, D-60629 Frankfurt, Germany \ and \  African Institute for Mathematical Sciences in Senegal (AIMS Senegal), 
KM2, Route de Joal, B.P. 14 18. Mbour, S\'en\'egal.}

\email{feulefac@math.uni-frankfurt.de, \  \ pierre.a.feulefack@aims-senegal.org}

\address{Goethe-Universit\"{a}t Frankfurt, Institut f\"{u}r Mathematik.
	Robert-Mayer-Str. 10, D-60629 Frankfurt, Germany.}

\email{jarohs@math.uni-frankfurt.de}

\begin{document}

\begin{abstract}
	In this work we study nonlocal operators and corresponding spaces with a focus on operators of order near zero. We investigate the interior regularity of eigenfunctions and of weak solutions to the associated Poisson problem depending on the regularity of the right-hand side. Our method exploits the variational structure of the problem and we prove that eigenfunctions are of class $C^{\infty}$ if the kernel satisfies this property away from its singularity. Similarly in this case, if in the Poisson problem the right-hand is of class $C^{\infty}$, then also any weak solution is of class $C^{\infty}$.
\end{abstract}
\maketitle

{\footnotesize
	\begin{center}
	\end{center}
}

\section{Introduction and Main results} 

A crucial role in the investigation of differential operators is the study of eigenfunctions and corresponding eigenvalues, if they exist. In the classical case of the Laplacian $-\Delta$ in a bounded domain $\Omega$ in $\R^N$, it is well known that there exists a sequence of functions $u_n\in H^1_0(\Omega)$, $n\in \N$ and corresponding values $\lambda_{n}>0$ such that 
$$
-\Delta u_n=\lambda_n u_n \quad\text{in $\Omega\quad$ and}\quad u_n=0\quad\text{in $\partial \Omega$.}
$$
Here, $H^1_0(\Omega)$ is as usual the closure of $C^{\infty}_c(\Omega)$ with respect to the norm $u\mapsto \Big(\|u\|_{L^2(\Omega)}^2+\||\nabla u|\|_{L^2(\Omega)}^2\Big)^{\frac12}$.\\
With the well-known De Giorgi iteration in combination with the Sobolev embedding, it follows that $u_n$ must be bounded and by a boot strapping argument using the regularity theory of the Laplacian it follows that $u_n$ is smooth in $\Omega$. In the model case of a nonlocal problem, one usually studies the \textit{fractional Laplacian} $(-\Delta)^s$ with $s\in(0,1)$. This operator can be defined via its Fourier symbol $|\cdot|^{2s}$ and it can be shown that for $\phi\in C^{\infty}_c(\R^N)$ we have
$$
(-\Delta)^s\phi(x)=c_{N,s}p.v.\int_{\R^N}\frac{u(x)-u(y)}{|x-y|^{N+2s}}\ dy=c_{N,s}\lim_{\eps\to 0^+}\int_{\R^N\setminus B_{\epsilon}(x)}\frac{u(x)-u(y)}{|x-y|^{N+2s}}\ dy, \quad x\in \R^N
$$
with a suitable normalization constant $c_{N,s}>0$. As in the above classical case, it can be shown that there exists a sequence of functions $u_n\in \cH^s_0(\Omega)$, $n\in \N$ and corresponding values $\lambda_{s,n}>0$ such that 
$$
(-\Delta)^s u_n=\lambda_{s,n} u_n \quad\text{in $\Omega\quad$ and}\quad u_n=0\quad\text{in $\R^N\setminus\Omega$.}
$$
Here, the space $\cH^s_0(\Omega)$ is given by the closure of $C^{\infty}_c(\Omega)$ ---understood as functions on $\R^N$--- with respect to the norm  $u\mapsto \Big(\|u\|_{L^2(\Omega)}^2+\cE_s(u,u)\Big)^{\frac12}$, where for $u,v\in C^{\infty}_c(\R^N)$ we set
$$
\cE_s(u,v):=\frac{c_{N,s}}{2}\int_{\R^N}\int_{\R^N} \frac{(u(x)-u(y))(v(x)-v(y))}{|x-y|^{N+2s}}\ dxdy.
$$
With similar methods as in the classical case, it follows that $u_n$ is smooth in the interior of $\Omega$.

\medskip

In the following we investigate the above discussion to the case where the \textit{kernel function} $z\mapsto c_{N,s}|z|^{-N-2s}$ is replaced by a measurable function $j:\R^N\to[0,\infty]$ such that for some $\sigma\in(0,2]$ we have
\begin{equation}
\label{assumptionsj}
j(z)=j(-z)\quad\text{for all $z\in \R^N$\quad and} \quad \int_{\R^N}\min\{1,|z|^{\sigma}\}j(z)\ dz<\infty.
\end{equation}
With $\sigma=2$, the above yields that $j\,dz$ is a L\'evy measure and the associated operator to this choice of kernel is of order below $2$. In the following we focus on the case where the singularity of $j$ is not too large, that is on the case $\sigma<1$ so that the associated operator is of order strictly below one. We call the operator in this case also of \textit{small order}.\\
Motivated by applications using nonlocal models, where a small order of the operator captures the optimal efficiency of the model \cite{PV18,AB17}, nonlocal operators with possibly order near zero, i.e. if \eqref{assumptionsj} is satisfied for all $\sigma>0$, have been studied in linear and nonlinear integro-differential equations, see \cite{FKV13,CP17,CW18,FJW20,JSW20,HS2121,F21,TW21} and the references in there. From a stochastic point of view, general classes of nonlocal operators appear as the generator of jump processes, where the jump behavior is modeled through types of \textit{L\'evy measures} and properties of \textit{associated harmonic} functions have been studied, see \cite{M14,G14,KM17,GK18} and the references in there. In particular, operators of the form $\phi(-\Delta)$ for a certain class of functions $\phi$ are of interest from a stochastic and analytic point of view, see e.g.\cite{BJ20,BVW21} and the references in there.\\

\medskip

In the following, we aim at investigating properties of bilinear forms and operators associated to a kernel $j$ satisfying \eqref{assumptionsj} for some $\sigma\in (0,2]$ from a variational point of view. For this, some further assumptions on $j$ are needed in our method and we present certain explicit examples at the end of this introduction, where our results apply.

\medskip

Let $\Omega\subset \R^N$ open, $u,v\in C^{0,1}_c(\Omega)$ understood as functions defined on $\R^N$, and consider the bilinear form
\begin{equation}\label{defi-bilinear}
b_{j,\Omega}(u,v):=\frac{1}{2}\int_{\Omega}\int_{\Omega}(u(x)-u(y))(v(x)-v(y))j(x-y)\ dxdy,\\
\end{equation}
where we also write $b_{j}(u,v):=b_{j,\R^N}(u,v)$ and $b_{j,\Omega}(u):=b_{j,\Omega}(u,u)$, $b_j(u)=b_{j}(u,u)$ resp. We denote 
\begin{equation}
\label{defi-space1}
D^j(\Omega):=\{u\in L^2(\Omega)\;:\; b_{j,\Omega}(u)<\infty\}.
\end{equation}
Associated to $b_j$ there is a \textit{nonlocal} self-adjoint operator $I_j$ which for $u,v\in C^{0,1}_c(\Omega)$ satisfies
	\begin{equation}\label{defi-op}
	b_{j}(u,v)=\int_{\R^N}I_{j}u(x)v(x)\ dx\quad\text{and is represented by}\quad I_{j}u(x)=p.v.\int_{\R^N}(u(x)-u(y))j(x-y)\ dy,\quad x\in \Omega.
	\end{equation}

To investigate the eigenvalue problem for $I_j$, we further need the space
$$
\cD^j(\Omega):=\{u\in D^j(\R^N)\;:\; 1_{\R^N\setminus \Omega} u\equiv 0\}.
$$
Note that clearly $\cD^j(\R^N)=D^j(\R^N)$ and both $D^j(\Omega)$ and $\cD^j(\Omega)$ are Hilbert spaces with scalar products 
$$
\langle \cdot,\cdot\rangle_{D^j(\Omega)}:=\langle \cdot,\cdot \rangle_{L^2(\Omega)}+b_{j,\Omega}(\cdot,\cdot)  \quad \text{and}\quad \langle \cdot,\cdot\rangle_{\cD^j(\Omega)}:=\langle \cdot,\cdot \rangle_{L^2(\Omega)}+b_j(\cdot,\cdot)\quad\text{resp.}
$$	

Our first result concerns the eigenfunctions for the operator $I_j$.
\begin{thm}\label{eigenfunctionsj}
Let \eqref{assumptionsj} hold with $\sigma=2$ and assume $j$ satisfies additionally
\begin{equation}\label{assumption2j}
\int_{\R^N}j(z)\ dz=\infty.
\end{equation}
Let $\Omega\subset \R^N$ be a bounded open set. Then there exists a sequence $u_n\in \cD^j(\Omega)$, $n\in \N$ and values $\lambda_n$ such that
$$
b_j(u_n,v)=\lambda_n\int_{\Omega} u_n(x)v(x)\ dx \quad\text{for all $v\in \cD^j(\Omega)$, $n\in \N$}
\quad\text{and}\quad
0<\lambda_1<\lambda_2\leq \ldots \leq \lambda_n\to \infty\quad\text{for $n\to \infty$.}
$$
Here, we have
\begin{equation}
\lambda_1=\Lambda_1(\Omega):=\inf_{\substack{u\in \cD^j(\Omega)\\ u\neq 0}}\frac{b_j(u)}{\|u\|_{L^2(\Omega)}^2}\label{first-eigenvalue}
\end{equation}
and $u_1$ is unique up to a multiplicative constant ---that is, $\lambda_1$ is simple. Moreover, $u_1$ can be chosen to be positive in $\Omega$. Furthermore, the following statements hold.
\begin{enumerate}
\item  If in addition $j$ satisfies
\begin{equation}\label{assumption3j}
\int_{B_R(0)\setminus B_r(0)} j^2(z)\ dz<\infty\quad\text{for all $0<r<R$,}
\end{equation}
then $u_n\in L^{\infty}(\Omega)$ for every $n\in \N$ and there is $C=C(\Omega,j,n)>0$ such that
$$
\|u_n\|_{L^{\infty}(\Omega)}\leq C\|u_n\|_{L^2(\Omega)}.
$$
\item Assume \eqref{assumptionsj} holds with $\sigma<\frac12$, \eqref{assumption3j} holds, and there is $m\in \N\cup\{\infty\}$ such that the following holds: We have $j\in W^{l,1}(\R^N\setminus B_{\epsilon}(0)$ for every $l\in \N$ with $l\leq 2m$, and there is some constant $C_j>0$ such that
$$
|\nabla j(z)|\leq C_j|z|^{-1-\sigma-N}\quad \text{for all $0<|z|\leq 3$.}
$$
Then $u_n\in H^m_{loc}(\Omega)$. In particular, $u_n\in C^{\infty}(\Omega)$ if $m=\infty$. 
\end{enumerate}
\end{thm}	

\noindent

For the definition of the Sobolev spaces $W^{l,1}$ and $H^m$ we refer to Section \ref{sobolev}. The first part of Theorem \ref{eigenfunctionsj} indeed follows immediately from the results of \cite{JW19}. To show the boundedness, we emphasize that in our setting, there are no Sobolev embeddings available and thus it is not clear how to implement the approach via the De Giorgi iteration. We circumvent this, by generalizing the $\delta$-decomposition introduced in \cite{FJW20}. The proof of the regularity statement is inspired by the approach used in \cite{C17}, where the author studies regularity of solutions to equations involving nonlocal operators which are in some sense comparable to the fractional Laplacian and uses Nikol'skii spaces. We emphasize that some of our methods generalize to the situation where the operator is not translation invariant and maybe perturbed by a convolution type operator. We treat these in the present work, too, see e.g. Theorem \ref{thm-bdd-full-domain} below. Using a probabilistic and potential theoretic approach, a local smoothness of bounded harmonic solutions solving in a certain very weak sense $I_ju=0$ in $\Omega$, have been obtained in \cite[Theorem 1.7]{GK18} for radial kernel functions using the same regularity as we impose in statement (2) of Theorem \ref{eigenfunctionsj} (see also \cite{M14,G14}). See also \cite{GKL21} for related regularity properties of solutions.\\

\medskip

To present our generalization of the above mentioned $\delta$-decomposition, let us first note that the first equality in \eqref{defi-op} can be extended, see Section \ref{preliminaries}. For this, let $\cV^j(\Omega)$ denote the space of those functions $u:\R^N\to\R$ such that $u|_{\Omega}\in D^j(\Omega)$ and 
\begin{equation}
\sup_{x\in \R^N}\int_{\R^N\setminus B_r(x)} |u(y)|j(x-y)\ dy<\infty\quad\text{for all $r>0$.}
\end{equation}
Given $f\in L^2_{loc}(\Omega)$, we then call $u\in \cV^j(\Omega)$ a \textit{(weak) supersolution} of $I_ju=f$ in $\Omega$, if 
\begin{equation}\label{supersolution}
b_j(u,v)\geq \int_{\Omega} f(x)v(x)\ dx\qquad\text{for all $v\in C^{\infty}_c(\Omega)$.}
\end{equation}
In this situation, we also say that $u$ satisfies in weak sense $I_ju\geq f$ in $\Omega$. Similarly, we define \textit{subsolutions} and \textit{solutions}.\\
We emphasize that this definition of supersolution is larger than the one considered in \cite{JW19}. In the case where $\sigma<1$ in \eqref{assumptionsj} this allows via a density result to extend the weak maximum principles presented in \cite{JW19} as follows.
\begin{prop}[Weak maximum principle]
	\label{wmp}
	Assume \eqref{assumptionsj} is satisfied with $\sigma<1$ and assume that 
	\begin{equation}\label{assumption2a}
	\text{$j$ does not vanish identically on $B_r(0)$ for any $r>0$.}
	\end{equation}
	Let $\Omega\subset \R^N$ open and with Lipschitz boundary, $c\in L^{\infty}_{loc}(\Omega)$, and assume either
	\begin{enumerate}
		\item $c\leq 0$ or 
		\item $\Omega$ and $c$ are such that $\|c^{+}\|_{L^{\infty}(\Omega)}<\inf_{x\in \Omega} \int_{\R^N\setminus \Omega}j(x-y)\ dy$.
	\end{enumerate}
	If $u\in \cV^j(\Omega)$ satisfies in weak sense 
	$$
	I_ju\geq c(x)u\quad\text{ in $\Omega$, $u\geq 0$ almost everywhere in $\R^{N}\setminus \Omega$, and}\quad \liminf\limits_{|x|\to\infty}u(x)\geq 0,
	$$
	then $u\geq 0$ almost everywhere in $\R^N$.
\end{prop}

\begin{rem} 
\begin{enumerate}
\item The Lipschitz boundary assumption on $\Omega$ is a technical assumption for the approximation argument.
\item Note that Assumption \eqref{assumption2a} readily implies the positivity $\Lambda_1(\Omega)$ defined in \eqref{first-eigenvalue}, whenever $\Omega$ is an open set in $\R^N$ which is bounded in one direction, that is $\Omega$ is contained (after a rotation) in a strip $(-a,a)\times \R^{N-1}$ for some $a>0$ (see \cite{FKV13,JW20}).
\end{enumerate}
\end{rem}

Up to our knowledge Proposition \ref{wmp} is even new in the case of $j=|\cdot|^{-N-2s}$, that is, the case of the fractional Laplacian (up to a multiplicative constant). In this situation, it holds $\cV^j(\Omega)=H^s(\Omega)\cap \cL^1_s$, where
$$
\cL^1_s=\left\{ u\in L^1_{loc}(\R^N)\;:\; \int_{\R^N}\frac{|u(x)|}{1+|x|^{N+2s}}\ dy<\infty\right\},
$$
and the proposition can be reformulated as follows.
\begin{cor}\label{wmp-special}
Let $s\in(0,\frac12)$, $\Omega\subset \R^N$ open and with Lipschitz boundary, $c\in L^{\infty}_{loc}(\Omega)$, and assume either
	\begin{enumerate}
		\item $c\leq 0$ or 
		\item $\Omega$ and $c$ are such that $\|c^{+}\|_{L^{\infty}(\Omega)}<\Lambda_1(\Omega)$.
	\end{enumerate}
	If $u\in H^s(\Omega)\cap \cL^1_s$ satisfies in weak sense 
	$$
	(-\Delta)^su\geq c(x)u\quad\text{ in $\Omega$, $u\geq 0$ almost everywhere in $\R^{N}\setminus \Omega$, and}\quad \liminf\limits_{|x|\to\infty}u(x)\geq 0,
	$$
	then $u\geq 0$ almost everywhere in $\R^N$.
\end{cor}

Similarly to the extension of the weak maximum principle, we have also the following extension of the strong maximum principle presented in \cite{JW19} in the case $\sigma <1$.

\begin{prop}[Strong maximum principle]
	\label{smp}
	Assume \eqref{assumptionsj} is satisfied with $\sigma<1$ and assume that $j$ satisfies additionally \eqref{assumption2j}. Let $\Omega\subset \R^N$ open and $c\in L^{\infty}_{loc}(\Omega)$ with $\|c^+\|_{L^{\infty}(\Omega)}<\infty$. Moreover, let $u\in \cV^j(\Omega)$, $u\geq 0$ satisfy in weak sense $I_ju\geq c(x) u$ in $\Omega$. Then the following holds.
	\begin{enumerate}
		\item If $\Omega$ is connected, then either $u\equiv 0$ in $\Omega$ or $\essinf_Ku>0$ for any $K\subset\subset \Omega$.
		\item  $j$ given in \eqref{defi-j} satisfies $\essinf_{B_r(0)}j>0$ for any $r>0$, then either $u\equiv 0$ in $\R^N$ or $\essinf_Ku>0$ for any $K\subset\subset \Omega$.
	\end{enumerate}
\end{prop}

Our last results concern the Poisson problem associated to the operator $I_j$.


\begin{thm}\label{poissonj}
Let \eqref{assumptionsj} hold with $\sigma=2$ and assume $j$ satisfies additionally \eqref{assumption2a}. Let $\Omega\subset \R^N$ be a bounded open set. Then for any $f\in L^2(\Omega)$ there is a unique solution $u\in \cD^j(\Omega)$ of $I_ju=f$. Moreover, if $j$ satisfies \eqref{assumption2j} and \eqref{assumption3j} and $f\in L^{\infty}(\Omega)$, then also $u\in L^{\infty}(\Omega)$ and there is $C=C(\Omega,j)>0$ such that
$$
\|u\|_{L^{\infty}(\Omega)}\leq C\|f\|_{L^\infty(\Omega)}.
$$
Additionally, if \eqref{assumptionsj} holds with $\sigma<\frac12$, there is $m\in \N\cup\{\infty\}$ such that $j$ satisfies the assumptions in Theorem \ref{eigenfunctionsj}(2), and it holds $f\in C^{2m}(\overline{\Omega})$, then $u\in H^m_{loc}(\Omega)$. More precisely in this case, for every $\beta\in \N_0^N$, $|\beta|\leq m$ and $\Omega'\subset\subset \Omega$ there is $C=C(\Omega,\Omega',j,\beta)>0$ such that
$$
\|\partial^{\beta}u\|_{L^2(\Omega')}\leq C\|f\|_{C^{2m}(\Omega)}.
$$
In particular, if $m=\infty$ and $f\in C^{\infty}(\Omega)$, then $u\in C^{\infty}(\Omega)$.
\end{thm}	

Our approach to prove Theorem \ref{eigenfunctionsj} and Theorem \ref{poissonj} is uniformly by considering an equation of the form
$$
I_ju=h\ast u+\lambda u+f \quad\text{in $\Omega$}
$$
for $f\in L^2(\Omega)$, $h\in L^{1}(\R^N)\cap L^2(\R^N)$, and $\lambda\in \R$. Moreover, several of our results need $I_j$ not to be translation invariant and this setup is discussed in Section \ref{preliminaries}.

\subsection{Examples}
We close this introduction with some classes of operators covered by our results.
\begin{enumerate}
\item As introduced in \cite{CW18,JSW20,FJW20} the \textit{logarithmic Laplacian} 
\begin{equation}\label{eq:log-laplace}
\loglap \phi(x)=c_NP.V.\int_{B_1(0)}\frac{\phi(x)-\phi(x+y)}{|y|^{N}}\ dy - c_N\int_{\R^N\setminus B_1(0)}\frac{\phi(x+y)}{|y|^{N}}\ dy+\rho_N\phi(x),  
\end{equation}
appears as the operator with Fourier-symbol $-2\ln(|\cdot|)$ and can be seen as the formal derivative in $s$ of $(-\Delta)^s$ at $s=0$. Here 
\begin{equation}
  \label{eq:def-c-n-rho-n}
c_N=\frac{\Gamma(\frac{N}{2})}{\pi^{N/2}}= \frac{2}{|S^{N-1}|} \qquad \text{and}\qquad \rho_N :=2\ln(2)+\psi(\frac{N}{2}) - \gamma   
\end{equation}  
where $\psi:= \frac{\Gamma'}{\Gamma}$ denotes the digamma function and $\gamma:= -\psi(1)=-\Gamma'(1)$ is the Euler-Mascheroni constant. With $j(z)=c_N1_{B_1(0)}(z)|z|^{-N}$ the operator $\loglap$ can be seen as a bounded perturbation of the operator class discussed in the introduction. The following sections cover in particular this operator.
\item The logarithmic Schr\"odinger operator $(I-\Delta)^{\log}$ as in \cite{F21} is an integro-differential  operator with Fourier-symbol $\log(1+|\cdot|^2)$  and also appears  as the formal derivative in $s$ of the relativistic Schr\"odinger operator $(I-\Delta)^s$ at $s=0$,
\[
(I-\Delta)^{\log}u(x)= d_{N}P.V.\int_{\R^N}\frac{u(x)-u(x+y)}{|y|^{N}}\omega(|y|)\ dy,
\]
 where $d_N=\pi^{-\frac{N}{2}}$, $\omega(r)=2^{1-\frac{N}{2}}r^{\frac{N}{2}}K_{\frac{N}{2}}(r)$ and $K_{\nu}$ is  the modified Bessel function of the second kind with index $\nu$. More generally, operators with symbol $\log(1+|\cdot|^{\beta})$ for some $\beta\in(0,2]$ are studied in \cite{KM14}.
\item Finally, also nonradial kernels of the type considered in \cite{JW19} satisfy in particular the assumptions \eqref{assumptions} and \eqref{assumption2}. See also also \cite{M14,G14,KM14} and references in there. 
\end{enumerate}

The paper is organized as follows. In Section \ref{preliminaries} we collect some general results concerning the spaces used in this paper and resulting definitions of weak sub- and supersolutions. Section \ref{sec:dense} is devoted to show several density results, which then are used to show the Propositions \ref{wmp} and \ref{smp}. In Section \ref{sec:bdd} we present a general approach to show boundedness of solutions and in Section \ref{sec:reg} we give the proof of an interior $H^1$-regularity estimate for solutions from which we then deduce the interior regularity statement as claimed in Theorem \ref{eigenfunctionsj}(2) and Theorem \ref{poissonj}.

\medskip

\paragraph{Notation} In the remainder of the paper, we use the following notation. Let $U,V\subset \R^N$ be nonempty measurable sets, $x \in \R^N$ and $r>0$. We denote by $1_U: \R^N \to \R$ the characteristic function, $|U|$ the Lebesgue measure, and $\diam(U)$ the diameter of $U$. The notation $V \subset \subset U$ means that $\overline V$ is compact and contained in the interior of $U$. The distance between $V$ and $U$ is given by $\dist(V,U):= \inf\{|x-y|\::\: x \in V,\, y \in U\}$. Note that this notation does {\em not} stand for the usual Hausdorff distance.  If $V= \{x\}$ we simply write $\dist(x,U)$. We let $B_r(U):=\{x\in \R^N\;:\; \dist(x,U)<r\}$, so that $B_r(x):=B_r(\{x\})$ is the open ball centered at $x$ with radius $r$. We also put $B:=B_1(0)$ and $\omega_N:=|B|$. Finally, given a function $u: U\to \R$, $U \subset \R^N$, we let $u^+:= \max\{u,0\}$ and $u^-:=-\min\{u,0\}$ denote the positive and negative part of $u$, and we write $\supp\ u$ for the support of $u$ given as the closure in $\R^N$ of the set ${\{ x\in U\;:\; u(x)\neq 0\}}$.\\

\textbf{Acknowledgements.}~ 
This work is supported by DAAD and BMBF (Germany) within the project 57385104. The authors thank Mouhamed Moustapha Fall and Tobias Weth for helpful discussions. We also thank the anonymous referee for valuable suggestions.

\section{Preliminaries}\label{preliminaries}

In the following, we generalize the translation invariant setting of the introduction. For this and from now on, let $k:\R^N\times\R^N\to[0,\infty]$ be a measurable function satisfying for some $\sigma\in(0,2]$
\begin{equation}
\label{assumptions}
\begin{split}
k(x,y)=k(y,x)&\quad\text{for all $x,y\in \R^N$ and }\quad \sup_{x\in \R^N}\int_{\R^N}\min\{1,|x-y|^{\sigma}\}k(x,y)\ dy<\infty.
\end{split}
\end{equation}
We let $j:\R^N\to[0,\infty]$ be the symmetric lower bound of $k$ given by
	\begin{equation}
	\label{defi-j}
	j(z):=\essinf\{k(x,x\pm z)\;:\; x\in \R^N\}\quad\text{for $z\in \R^N$.}
	\end{equation}
	
For $\Omega\subset \R^N$ open let

\begin{align*}
b_{k,\Omega}(u,v)&:=\frac{1}{2}\int_{\Omega}\int_{\Omega}(u(x)-u(y))(v(x)-v(y))k(x,y)\ dxdy,\\
\kappa_{k,\Omega}(x)&:=\int_{\R^N\setminus \Omega}k(x,y)\ dy\in[0,\infty] \quad\text{for $x\in \R^N$, and}\\
K_{k,\Omega}(u,v)&:=\int_{\Omega} u(x)v(x)\kappa_{k,\Omega}(x)\ dx,
\end{align*}

where, if $u=v$ we put
$$
b_{k,\Omega}(u):=b_{k,\Omega}(u,u)\quad\text{and}\quad K_{k,\Omega}(u):=K_{k,\Omega}(u,u)
$$
and we drop the index $\Omega$, if $\Omega=\R^N$. Note that we have for any fixed $x\in \Omega$ that $\kappa_{k,\Omega}(x)<\infty$ by \eqref{assumptions}.
We consider the function spaces
\begin{align*}
D^k(\Omega)&:=\Big\{u\in L^2(\Omega)\;:\; b_{k,\Omega}(u)<\infty\Big\},\\
\cD^k(\Omega)&:=\Big\{u\in D^k(\R^N)\;:\; u=0 \ \text{on $\R^{N}\setminus \Omega$}\Big\},\\
\cV^k(\Omega)&:=\Big\{u:\R^N\to\R\;:\; u|_{\Omega}\in D^k(\Omega)\text{ and, for all $r>0$, }\sup_{x\in \R^N}\int_{\R^N\setminus B_r(x)}|u(y)|k(x,y)\ dy<\infty\Big\}\quad\text{and}\\
\cV_{loc}^k(\Omega)&:=\Big\{u:\R^N\to\R\;:\;  u|_{\Omega'}\in \cV^k(\Omega') \text{ for all $\Omega'\subset\subset\Omega$}\Big\}. 
\end{align*}

\begin{lemma}
\label{properties2}
Let $U\subset \Omega\subset \R^N$ open and $u:\R^N\to\R$. Then the following hold:
\begin{enumerate}
\item $u\in \cD^k(\Omega)\ \Rightarrow\ u|_{\Omega}\in D^k(\Omega)$.
\item $\cD^k(U)\subset \cD^k(\Omega)\subset \cV^k(\Omega)\subset \cV^k(U)\subset \cV^k_{loc}(U)$.
\end{enumerate}
\end{lemma}
\begin{proof}
This follows immediately from the definitions (see also \cite[Section 3]{JW19}).
\end{proof}

\begin{lemma}[see Proposition 3.3 in \cite{JW19} or Proposition 1.7 in \cite{JW20}]
\label{poincare}
For $\Omega\subset \R^N$ open, let $\Lambda_1(\Omega)$ be given by (c.f. \eqref{first-eigenvalue})
$$
\Lambda_1(\Omega):=\inf_{\substack{u\in \cD^j(\Omega)\\ u\neq 0}}\frac{b_j(u)}{\|u\|_{L^2(\Omega)}^2}
$$
and let 
$$
\lambda(r)=\inf\{\Lambda_1(\Omega)\;:\;\text{$\Omega\subset \R^N$ open with $|\Omega|=r$ }\}.
$$
Then $\lim\limits_{r\to\infty}\lambda(r)\geq \int_{\R^N}j(z)\ dz$ with $j(z):=\essinf\{k(x,x\pm z)\;:\; z\in \R^N\}$ for $z\in \R^N$ as in \eqref{defi-j}.
\end{lemma}

\begin{lemma}
\label{properties1}
Let $\Omega\subset \R^N$ open and let $X$ be any of the above function spaces. Then the following hold:
\begin{enumerate}
\item $b_{k,\Omega}$ is a bilinear form and in particular we have $b_{k,\Omega}(u,v)\leq b_{k,\Omega}^{1/2}(u)b_{k,\Omega}^{1/2}(v)$. Moreover, $D^k(\Omega)$ and $\cD^k(\Omega)$ are Hilbert spaces with scalar products 
\begin{align*}
\langle u,v\rangle_{D^k(\Omega)}&=\langle u,v\rangle_{L^2(\Omega)}+b_{k,\Omega}(u,v),\\
\langle u,v\rangle_{\cD^k(\Omega)}&=\langle u,v\rangle_{L^2(\Omega)}+b_{k,\R^N}(u,v).
\end{align*}
\item If $u\in X$, then $u^{\pm},|u|\in X$ and we have $b_{k,\Omega'}(u^+,u^-)\leq 0$ for all $\Omega'\subset \Omega$ with $b_{k,\Omega'}(u)<\infty$.
\item If $g\in C^{0,1}(\R^N)$, $u\in X$, then $g\circ u\in X$.
\item $C^{0,1}_c(\Omega)\subset X$.
\item $\phi\in C^{0,1}_c(\Omega)$, $u\in X$, then $\phi u\in \cD^k(\Omega)$, where if necessary we extend $u$ trivially to a function on $\R^N$. Moreover, there is $C=C(N,k,\|\phi\|_{C^{0,1}(\Omega)})>0$ such that 
$$
b_{k,\R^N}(\phi u)\leq C\Big(\|u\|_{L^2(\Omega')}^2+b_{k,\Omega'}(u)\Big)
$$
for any $\Omega'\subset\Omega$ with $\supp\,\phi\subset\subset \Omega'$.
\end{enumerate}
\end{lemma}
\begin{proof}
Theses statements follow directly from the definition (c.f. \cite[Section 3]{JW19}). To be precise in the last part, let $\phi\in C^{0,1}_c(\Omega)$ and fix $L:=\|\phi\|_{C^{0,1}(\Omega)}$. That is, we have
\[
|\phi(x)|\leq L\quad\text{and}\quad |\phi(x)-\phi(y)|\leq L|x-y|.
\]
Then using the inequality for $x,y\in \R^N$
\[
|\phi(x)u(x)-\phi(y)u(y)|^2\leq 2|\phi(x)-\phi(y)|^2|u(x)|^2+2|\phi(y)|^2|u(x)-u(y)|^2
\]
we find by the assumptions \eqref{assumptions}
\begin{align*}
b_{k,\R^N}(\phi u)&\leq b_{k,\Omega'}(\phi u)+L^2\int_{\supp\,\phi}|u(x)|^2\kappa_{k,\Omega'}(x)\ dx\\
&\leq 2L^2\int_{\Omega'}\int_{\Omega'}|u(x)|^2|x-y|^2k(x,y)\ dy\ dx+2L^2b_{k,\Omega'}(u)+L^2\sup_{x\in \supp\,\phi}\kappa_{k,\Omega'}(x)\|u\|_{L^2(\Omega')}^2\\
&\leq 2L^2\Big(\sup_{x\in \Omega'}\int_{\Omega'} |x-y|^2k(x,y)\ dy+\sup_{x\in \supp\,\phi}\kappa_{k,\Omega'}(x)\Big)\|u\|_{L^2(\Omega')}^2+2L^2b_{k,\Omega'}(u)<\infty.
\end{align*}
\end{proof}

\begin{rem}
\label{bilinearform rewrite}
\begin{enumerate}
\item Note that for $u,v\in \cD^k(\Omega)$ we have 
\[
b_{k}(u,v)=b_{k,\R^N}(u,v)=b_{k,\Omega}(u,v)+K_{k,\Omega}(u,v).
\]
\item It follows in particular that there is a nonnegative self-adjoint operator $I_k$ associated to $b_{k,\R^{N}}=b_k$ as mentioned in the introduction.
\end{enumerate}
\end{rem}

\begin{lemma}
\label{well-defined}
Let $\Omega\subset \R^N$ open and $u\in \cV_{loc}^k(\Omega)$. Then $b_{k}(u,\phi)$ is well-defined for any $\phi \in C^{\infty}_c(\Omega)$.
\end{lemma}
\begin{proof}
Let $\phi\in C^{\infty}_c(\Omega)$ and fix $U\subset \subset\Omega$ such that $\supp\ \phi\subset\subset U$. Then with the symmetry of $k$
\begin{align*}
|b_{k}(u,\phi)|&\leq |b_{k,U}(u,\phi)|+\int_{U} |\phi(x)|\int_{\R^N\setminus U} |u(x)-u(y)|k(x,y)\ dydx\\
&\leq b_{k,U}^{1/2}(u)b_{k,U}^{1/2}(\phi)+\int_{\supp\,\phi}|\phi(x)| \ dx\sup_{x\in \R^N}\int_{\R^N\setminus B_{\epsilon}(x)}|u(y)|k(x,y)\ dy\\
&\qquad\qquad\qquad +\int_{\supp\,\phi}|\phi(x)u(x)| \ dx \sup_{x\in \R^N}\int_{\R^N\setminus B_{\epsilon}(x)}k(x,y)\ dy<\infty,
\end{align*}
where $\epsilon=\dist(\supp\ \phi,\R^N\setminus U)>0$.
\end{proof}
 
\begin{defi}
	\label{defi-weak}
Let $\Omega\subset \R^N$ open and $f\in L^1_{loc}(\Omega)$. Then $u\in \cV_{loc}^k(\Omega)$ is called a \textit{weak supersolution} of $I_ku=f$ in $\Omega$, if
\[
b_{k}(u,\phi)\geq \int_{\Omega} f(x)\phi(x)\ dx\quad\text{for all nonnegative $\phi\in C^{\infty}_c(\Omega)$.}
\]
We also say that \textit{$u$ satisfies $I_ku\geq f$ weakly in $\Omega$}.\\
Similarly, we define \textit{weak subsolutions} and \textit{solutions}.
\end{defi}

\begin{rem}
	\begin{enumerate}
		\item We note that by Assumption \ref{assumptions}, it follows that for any function $u\in \cV^k_{loc}(\Omega)$ with $u|_{\Omega}\in C^{0,1}(\Omega)$ for $\Omega\subset \R^N$ open, we have $I_ku|_U\in L^{\infty}(U)$ for any $U\subset\subset \Omega$ and
$$
	I_ku(x)=\int_{\R^N}(u(x)-u(y))k(x,y)\ dy\quad\text{for $x\in \Omega$.}
$$
This follows similarly to the proof of the statements in Lemma \ref{properties1}.
\item If $u\in \cV^k(\Omega)$, then indeed also $b_k(u,\phi)$ is well-defined for all $\phi\in \cD^k(\Omega)$. Hence also $b_k$ is well defined on $\cV^k_{loc}(\Omega)\times \cD^k(U)$ for all $U\subset\subset \Omega$. In some of our results the statements need a Lipschitz-boundary of $\Omega$, which comes into play due to approximation with $C^{\infty}_c(\Omega)$-functions (see Section \ref{sec:dense} below). However, this can be weakened, if $u\in \cV^k(\Omega)$ and the space of test-functions is adjusted.
\end{enumerate}
\end{rem}

\begin{lemma}
	\label{localization}
	Let $\Omega\subset \R^N$ open. Let $\Omega_1\subset\subset\Omega_2\subset\subset \Omega_3\subset\subset \Omega$. Let $\eta\in C^{0,1}_c(\R^N)$ such that $0\leq \eta\leq 1$ in $\R^N$ and we have
	$$
	\eta=1\quad\text{in $\Omega_{2}\quad$ and }\quad \eta=0\quad\text{in $\R^N\setminus \Omega_3$.}
	$$
	Let $f\in L^1_{loc}(\Omega)$ and let $u\in \cV^k_{loc}(\Omega)$ satisfy in weak sense $Iu\geq f$ in $\Omega$. Then the function $v=\eta u\in \cD^k(\Omega_3)$ satisfies in weak sense $Iv\geq f+g_{\eta,u}(x)$ in $\Omega_1$, where
	$$
	g_{\eta,u}(x)=\int_{\R^N\setminus \Omega_2}(1-\eta(y))u(y)k(x,y)\ dy\quad\text{for $x\in \Omega_1$.}
	$$
\end{lemma}
\begin{proof}
	The fact, that $v\in \cD^k(\Omega_3)$ follows from Lemma \ref{properties1}. Let $\phi\in C^{\infty}_c(\Omega_1)$, then
	$$
	\int_{\R^N}v(x)I_k\phi(x)\ dx\geq \int_{\R^N}f(x)\phi(x)\ dx-\int_{\R^N}(1-\eta(x))u(x)I_k\phi(x)\ dx.
	$$
	Here, since $(1-\eta)u\equiv 0$ on $\Omega_2$, we have
	\begin{align*}
	\int_{\R^N}(1-\eta(x))u(x)I_k\phi(x)\ dx&=\int_{\R^N}\phi(x) [I_k(1-\eta)u](x)\ dx\\
	&=-\int_{\Omega_1}\phi(x) \int_{\R^N\setminus \Omega_2}(1-\eta(y))u(y)k(x,y)\ dy\ dx.
	\end{align*}
	Thus the claim follows.
\end{proof}

\begin{rem}
	The same result as in Lemma \ref{localization} also holds if ``$\geq$'' in the solution type is replaced by ``$\leq$'' or ``$=$''.
\end{rem}

In the following, it is useful to understand functions $u\in D^k(\Omega)$ satisfying $b_{k,\Omega}(u)=0$.

\begin{prop}
	\label{constant}
	Assume that the symmetric lower bound $j$ of $k$ defined in \eqref{defi-j} satisfies $\int_{\R^N} j(z)\ dz=\infty$. Let $\Omega\subset \R^N$ open and bounded and let $u\in D^k(\Omega)$ such that $b_{k,\Omega}(u)=0$. Then $u$ is constant.
\end{prop}
\begin{proof}
	Let $x_0\in \Omega$ and fix $r>0$ such that $B_{2r}(x_0)\subset \Omega$. Denote $q(z):=\min\{c,j(z)\}1_{B_r(0)}(z)$, where we may fix $c>0$ such that $|\{q>0\}|>0$ due to the assumption on $j$. Then by Lemma \ref{iterationlemma} we have
	\[
	0=2b_{k,\Omega}(u)\geq 2b_{q,\Omega}(u)\geq \frac{1}{2\|q\|_{L^1(\R^N)}} b_{q\ast q,B_{r}(x_0)}(u),
	\]
	where $a\ast b=\int_{\R^N}a(\cdot-y)b(y)\ dy$ denotes as usual the convolution.	Note that since $q\in L^1(\R^N)\cap L^{\infty}(\R^N)$ with $q=0$ on $\R^N\setminus B_r(0)$, it follows that $q\ast q\in C(\R^N)$ with support in $B_{2r}(0)$ and we have
	\[
	q\ast q(0)=\int_{\R^N}q(z)^2\ dz>0
	\]
	by the assumption on $j$. Hence there is $R>0$ with $q\ast q\geq \epsilon$ for some $\epsilon>0$ and thus we have
	\[
	0=b_{q\ast q,B_{r}(x_0)}(u)\geq b_{q\ast q,B_{\rho}(x_0)}(u)\geq \frac{\epsilon}{2}\int_{B_{\rho}(x_0)}\int_{B_{\rho}(x_0)}(u(x)-u(y))^2\ dxdy.
	\]
	for any $\rho\in(0,\frac{R}{2}]$. But then $u(x)=u(y)$ for almost every $x,y\in B_{R/2}(x_0)$ so that $u$ is constant a.e. in $B_{\rho}(x_0)$. Since $b_{k,\Omega}(u)=b_{k,\Omega}(u-m)$ for any $m\in \R$, we may next assume that $u=0$ in $B_{R/2}(x_0)$ and show that indeed we have $u=0$ a.e. in $\Omega$. Denote by $W$ the set of points $x\in \Omega$ such that there is $r>0$ with $u=0$ a.e. in $B_r(x)$. By definition $W$ is open and the above shows that $W$ is nonempty. Next, let $(x_n)_n\subset W$ be a sequence with $x_n\to x\in \Omega$ for $n\to\infty$. Then there is $r_x>0$ such that $B_{4r_x}(x)\subset \Omega$ and we can find $n_0\in \N$ such that $x\in B_{r_x}(x_n)\subset B_{2r_x}(x_n)\subset \Omega$ for $n\geq n_0$. Repeating the above argument, it follows that $u$ must be zero in $B_{r_x}(x_n)$ and thus $x\in W$. Hence, $W$ is relatively open and closed in $\Omega$ and since $W$ is nonempty, we have $W=\Omega$. That is $u=0$ in $\Omega$.
\end{proof}

\subsection{On Sobolev and Nikol'skii spaces}\label{sobolev}
We recall here the notations and properties of Sobolev and Nikol'skii spaces as introduced in \cite{T95,C17}. In the following, let $p\in[1,\infty)$ and $\Omega\subset \R^N$ open. 

\subsubsection{Sobolev spaces} If $k\in \N_0$, we set as usual
$$
W^{k,p}(\Omega):=\Big\{u\in L^p(\Omega)\;:\; \text{ $\partial^{\alpha}u$ exists for all $\alpha\in \N_0^{n}$, $|\alpha|\leq k$ and belongs to $L^p(\Omega)$ }\Big\}
$$
for the Banach space of $k$-times (weakly) differentialable functions in $L^p(\Omega)$. Moreover, as usual, for $\sigma\in(0,1)$, $p\in[1,\infty)$ we set
$$
W^{\sigma,p}(\Omega):=\Big\{u\in L^p(\Omega)\;:\; \frac{u(x)-u(y)}{|x-y|^{\frac{n}{p}+\sigma}}\in L^{p}(\Omega\times \Omega)\Big\}.
$$
With the norm
$$
\|u\|_{W^{\sigma,p}(\Omega)}=\|u\|_{L^p(\Omega)}^p+[u]_{W^{\sigma,p}(\Omega)},\quad\text{
	where}\quad
[u]_{W^{\sigma,p}(\Omega)}=\Bigg(\iint_{\Omega\times \Omega}\frac{|u(x)-u(y)|^p}{|x-y|^{n+\sigma p}}\ dxdy\Bigg)^{1/p}
$$
the space $W^{\sigma,p}(\Omega)$ is a Banach space. For general $s=k+\sigma$, $k\in\N_0$, $\sigma\in[0,1)$ the Sobolev space is defined as
$$
W^{s,p}(\Omega):=\Big\{u\in W^{k,p}(\Omega)\;:\; \partial^{\alpha}u\in W^{\sigma,p}(\Omega)\ \text{for all $\alpha\in \N_0^{n}$ with $|\alpha|= k$ }\Big\}.
$$
Finally, in the particular case $p=2$ the space $H^{s}(\Omega):=W^{s,2}(\Omega)$ is a Hilbert space.

\subsubsection{Nikol'skii spaces} For $u:\Omega\to\R$ and $h\in \R$, let $\Omega_h:=\{x\in \Omega\;:\; \dist(x,\partial\Omega)>h\}$ and, with $e\in \partial B_1(0)$, we let
$$
\delta_hu(x)=\delta_{h,e}u(x):=u(x+he)-u(x).
$$
Moreover, for $l\in \N$, $l>1$ let
$$
\delta_{h}^{l}u(x)=\delta_h(\delta_h^{l-1}u)(x).
$$
For $s=k+\sigma>0$ with $k\in \N_0$ and $\sigma\in(0,1]$ define
$$
N^{s,p}(\Omega):=\Big\{u\in W^{k,p}(\Omega)\;:\; [\partial^{\alpha}u]_{N^{\sigma,p}(\Omega)}<\infty \ \text{for all $\alpha\in \N_0^{n}$ with $|\alpha|= k$ }\Big\}, 
$$
where
$$
[u]_{N^{\sigma,p}(\Omega)}=\sup_{\substack{e\in \partial B_1(0)\\ h>0}}h^{-\sigma}\|\delta_{h,e}^2u\|_{L^p(\Omega_{2h})}.
$$
It follows that $N^{s,p}(\Omega)$ is a Banach space with norm $\|u\|_{N^{s,p}(\Omega)}:=\|u\|_{W^{k,p}(\Omega)}+\sum_{|\alpha|=k}[\partial^{\alpha}u]_{N^{\sigma,p}(\Omega)}$. It can be shown that this norm is equivalent to
$$
\|u\|_{L^p(\Omega)}+\sum_{|\alpha|=k}\sup_{\substack{e\in \partial B_1(0)\\ h>0}}h^{m-\sigma}\|\delta_{h,e}^lu\|_{L^p(\Omega_{lh})}
$$
for any fixed $m,l\in \N_0$ with $m<\sigma$ and $l>\sigma-m$ (see \cite[Theorem 4.4.2.1]{T95}). 
\begin{prop}[see e.g. Propositions 3 and 4 in \cite{C17}]
	\label{prop-sobolev-nikolskii}
	Let $\Omega\subset \R^N$ open and with $C^{\infty}$ boundary. Moreover, let $t>s>0$ and $1\leq p<\infty$. Then
	$$
	N^{t,p}(\Omega)\subset W^{s,p}(\Omega)\subset N^{s,p}(\Omega).
	$$
\end{prop}

\section{Density results and Maximum principles}\label{sec:dense}

The main goal of this section is to show the following.
\begin{thm}\label{main-thm-density}
	Let either $\Omega=\R^N$ or $\Omega\subset \R^N$ open and bounded with Lipschitz boundary. In the following, let $X(\Omega):=\cD^k(\Omega)$ or $D^k(\Omega)$. Then $C^{\infty}_c(\Omega)$ is dense in $X(\Omega)$. Moreover, if $u\in X(\Omega)$ is nonnegative, then we have
	\begin{enumerate}
		\item There exists a sequence $(u_n)_n\subset X(\Omega)\cap L^{\infty}(\Omega)$ with $\lim\limits_{n\to\infty}u_n= u$ in $X(\Omega)$ satisfying that for every $n\in \N$ there is $\Omega'_n\subset\subset\Omega$ with $u_n=0$ on $\Omega\setminus \Omega'_n$ and $0\leq u_n\leq u_{n+1}\leq u$.
		\item There exists a sequence $(u_n)_n\subset C^{\infty}_c(\Omega)$ with $u_n\geq0$ for every $n\in \N$ and $\lim\limits_{n\to\infty}u_n= u$ in $X(\Omega)$.	
	\end{enumerate}
\end{thm}
\begin{rem}
	To put Theorem \ref{main-thm-density} into perspective, we consider the following examples.
	\begin{enumerate}
		\item In the case $k(x,y)=|x-y|^{-2s-N}$ for some $s\in(0,\frac{1}{2})$, the above Theorem is well-known and leads to the interesting property that for any open, bounded Lipschitz set $\Omega\subset \R^N$ we have
		$$
		D^k(\Omega)=H^s(\Omega)=H^s_0(\Omega).
		$$
		We emphasize that the above equality also holds for $s=\frac{1}{2}$. Moreover, if $s<\frac{1}{2}$, it also holds $H^s(\Omega)=\{u\in H^s(\R^N)\;:\; 1_{\R^N\setminus\Omega}u\equiv 0\}$.
		\item If $k(x,y)=1_{B_1(0)}(x-y)|x-y|^{-N}$, $D^k(\Omega)$ is associated to the function space of the \textit{localized} logarithmic Laplacian (see \cite{CW18}).
	\end{enumerate}
\end{rem}
The proof is split into several smaller steps. Recall that $\cD^k(\R^N)=D^k(\R^N)$ by definition.

\begin{lemma}
\label{density-pre1}
	Let $u\in D^k(\R^N)$. Then there is a sequence $(u_n)_n\subset D^k(\R^N)$ with $\lim\limits_{n\to\infty}u_n= u$ in $D^k(\R^N)$ satisfying that for every $n\in \N$ there is $\Omega_n\subset\subset\R^N$ with $u_n=0$ on $\R^N\setminus \Omega_n$. Moreover, if $u\geq 0$, then $(u_n)_n$ can be chosen to satisfy in addition $0\leq u_n\leq u_{n+1}\leq u$.
\end{lemma}
\begin{proof}
	For $n\in \N$ let $\phi_n\in C^{0,1}_c(\R^N)$ be radially symmetric and such that $\phi_n\equiv 1$ on $B_n(0)$, $\phi_n\equiv 0$ on $B_{n+1}(0)^c$. Clearly, we may assume that $[\phi_n]_{C^{0,1}(\R^N)}=1$. By Lemma \ref{properties1} there is hence some  $C=C(N,k)>0$ with $b_{k,\R^N}(\phi_n u)\leq C\|u\|_{D^k(\R^N)}$ for all $n\in \N$. In the following, let $u_n:=\phi_nu$ and without loss of generality we may assume $u\geq 0$. Since then $0\leq u-u_n\leq u$ on $\R^N$ and $u-u_n=0$ on $B_n$, by dominated convergence we have $\lim\limits_{n\to\infty}\|u-u_n\|_2=0$. Moreover, by choice of $\phi_n$ we have for $x,y\in \R^N$
	\begin{align*}
	|u(x)(1-\phi_n(x))-u(y)(1-\phi_n(y))|&\leq |u(x)-u(y)|(1-\phi_n(x)) +|u(y)||\phi_n(x)-\phi_n(y)|\\
	&\leq |u(x)-u(y)| +|u(y)|\min\{1,|x-y|\}=:U(x,y).
	\end{align*}
	Here, $U(x,y)\in L^2(\R^N\times \R^N,k(x,y)\ d(x,y))$, since
	\begin{align*}
	\iint_{\R^N\times\R^N} U(x,y)k(x,y)\ dxdy&=b_{k,\R^N}(u)+\int_{\R^N}|u(y)|^2\int_{\R^N} \min\{1,|x-y|^{2}\}k(x,y)\ dxdy\\
	&\leq b_{k,\R^N}(u)+\int_{\R^N}|u(y)|^2\ dy\,\supp_{x\in \R^N}\int_{\R^N} \min\{1,|x-y|^{2}\}k(x,y)dy<\infty.
	\end{align*}
	Thus $\lim\limits_{n\to\infty}b_{k,\R^N}(u-u_n)=0$ by the dominated convergence Theorem.
\end{proof}

\begin{prop}
\label{density1}
We have that $C^{\infty}_c(\R^N)$ is dense in $D^k(\R^N)$. Moreover, if $u\in D^k(\R^N)$ is nonnegative, then there exists $(\phi_n)_n\subset C^{\infty}_c(\R^N)$ with $\phi_n\geq0$ for every $n\in \N$ and $\lim\limits_{n\to\infty}\phi_n= u$ in $D^k(\R^N)$.
\end{prop}
\begin{proof}
Let $u\in D^k(\R^N)$. Moreover, let $\phi_n\in C^{0,1}_c(\R^N)$ for $n\in \N$ be given by Lemma \ref{density-pre1} such that $\|u-\phi_n u\|_{s,p}<\frac{1}{n}$. Then $v_n:=\phi_n u\in D^k(\R^N)$ and there is $R_n>0$ with $v_n\equiv 0$ on $\R^N\setminus B_{R_n}(0)$. Next, let $(\rho_\eps)_{\eps\in(0,1]}$ by a Dirac sequence and denote $v_{n,\eps}:=\rho_\eps\ast v_n$. Then $v_{n\eps}\in C^{\infty}_c(\R^N)$ for all $n\in \N$, $\eps\in(0,1]$ and
	\[
	b_{k,\R^N}(u-v_{n,\eps})\leq b_{k,\R^N}(u-v_n)+ b_{k,\R^N}(v_n-v_{n,\eps})\leq \frac{1}{n}+ b_{k,\R^N}(v_n-v_{n,\eps}).
	\]
	It is hence enough to show that $v_{n,\eps}\to v_n $ in $D^k(\R^N)$ for $\eps\to 0$. In the following, we write $v$ in place of $v_n$ and $v_\eps=\rho_\eps\ast v$ in place of $v_{n,\eps}$ for $\eps\in(0,1]$. Moreover, let $R=R_n>0$ with $v=v_n=0$ on $\R^N\setminus B_R(0)$. Clearly, $v_{\eps}\to v$ in $L^2(\R^N)$ for $\eps\to 0$ and this convergence is also pointwise almost everywhere. Hence it is enough to analyze the convergence of $b_{k,\R^N}(v-v_{\eps})$ as $\epsilon\to 0$. From here, the proof follows along the lines of \cite[Proposition 4.1]{JW19} noting that there it is not used that $k$ only depends on the difference of $x$ and $y$. Note here, that if $u$ is nonnegative then the above constructed sequence is also nonnegative.
\end{proof}

\begin{lemma}
\label{consequences of assumption2}
Let $\Omega\subset \R^N$ open and such that $\partial\Omega$ is bounded. Denote $\delta(x):=\dist(x,\R^N\setminus \Omega)$. Then the following is true.
\begin{enumerate}
\item There is $C=C(N,\Omega,k)>0$ such that $\kappa_{k,\Omega}(x)\leq C\delta^{-\sigma}(x)\ $ for $x\in \Omega$.
\item If $\Omega$ is bounded, then $1_{\Omega}\in D^k(\R^N)$.
\end{enumerate}
\end{lemma}
\begin{proof}
Let $C=C(N,\Omega,k)>0$ be constants varying from line to line and denote $U:=\{x\in \R^N\;:\; \dist(x,\Omega)\leq 1\}$. To see item 1., let $x\in \Omega$ and fix $p\in \partial\Omega$ such that $\delta(x)=|x-p|$. Then
\begin{align*}
\kappa_{k,\Omega}(x)\leq C+\int_{U\setminus \Omega}\frac{|x-p|^{\sigma}}{|x-p|^{\sigma}}k(x,y)\ dy\leq C+\delta(x)^{-\sigma}\int_{U\setminus \Omega} |x-y|^{\sigma}k(x,y)\ dy\leq C\delta^{-\sigma}(x),
\end{align*}
where we have used that $|x-p|\leq |x-y|$ for $y\in \R^N\setminus \Omega$. Now 2. follows immediately from 1., since we have
$$
b_{k,\R^N}(1_{\Omega})=\int_{\Omega}\int_{\R^N\setminus\Omega}k(x,y)\ dydx\leq C \int_{\Omega} \delta^{-\sigma}(x)\ dx <\infty.
$$
\end{proof}

\begin{thm}[See Theorem \ref{main-thm-density}]
\label{density2}
Let $\Omega\subset \R^N$ be an open bounded set with Lipschitz boundary. Then $C^{\infty}_c(\Omega)$ is dense in $\cD^k(\Omega)$. Moreover, if $u\in \cD^k(\Omega)$ is nonnegative, then we have
\begin{enumerate}
\item There exists a sequence $(u_n)_n\subset \cD^k(\Omega)$ with $\lim\limits_{n\to\infty}u_n= u$ in $\cD^k(\Omega)$ satisfying that for every $n\in \N$ there is $\Omega'_n\subset\subset\Omega$ with $u_n=0$ on $\R^N\setminus \Omega'_n$ and $0\leq u_n\leq u_{n+1}\leq u$.
\item There exists a sequence $(u_n)_n\subset C^{\infty}_c(\Omega)$ with $u_n\geq0$ for every $n\in \N$ and $\lim\limits_{n\to\infty}u_n= u$ in $\cD^k(\Omega)$.
\end{enumerate}
\end{thm}

\begin{proof}
Note that the second claim follows immediately from the first one using \cite[Proposition 4.1]{JW19} as in the proof of Proposition \ref{density1}. Then also the main claim follows by considering $u^{\pm}$ separately. Hence it is enough to show 1. We proceed similar to \cite[Theorem 3.1]{CW18}. Denote $\delta(x):=\dist(x,\R^N\setminus \Omega)$. For $r>0$, define the Lipschitz map
\[
\phi_r:\R^N\to\R,\quad\phi_r(x)=\left\{\begin{aligned}
&0 &&\delta(x)\geq 2r,\\
&2-\frac{\delta(x)}{r} && r\leq \delta(x)\leq 2r,\\
&1&& \delta(x)\leq r.
\end{aligned}\right.
\]
Note that we have $\phi_s\leq \phi_r$ for $0<s\leq r$. We show
\begin{equation}
\label{claim for density2}
u\phi_r\in \cD^k(\Omega)\ \text{ for $r>0$ sufficiently small and } b_{k,\R^N}(u\phi_r)\to 0 \ \text{ for $r\to 0$.}
\end{equation}
Note that once this is shown, we have $u(1-\phi_r)\in \cD^k(\Omega)$ for $r>0$ sufficiently small and $u(1-\phi_r)\to u$ for $r\to0$. Since also $0\leq u(1-\phi_r)\leq u(1-\phi_s)$ for $0<s\leq r$ and $u(1-\phi_r)=0$ for $x\in \R^N$ with $\delta(x)\leq r$, it follows that \eqref{claim for density2} implies 1.\\
The remainder of the proof is to show \eqref{claim for density2}. For this, let $C=C(N,\Omega,k)>0$ be a constant which may vary from line to line. Let $A_t:=\{x\in \Omega\;:\; \delta(x)\leq t\}$. Note that $u\phi_r$ vanishes on $\R^N\setminus A_{2r}$, we have $0\leq \phi_r\leq 1$ and, moreover,
\[
|\phi_r(x)-\phi_r(y)|\leq  \min\Big\{C\frac{|x-y|}{r},1\Big\}\quad\text{for $x,y\in \R^N$.}
\]
Then proceeding similarly to the proof of Lemma \ref{properties1}.(5) we find for $r$ small enough
\begin{align*}
b_{k,\R^N}&(u\phi_r)=\frac{1}{2}\int_{A_{4r}}\int_{A_{4r}}(u(x)\phi_r(x)-u(y)\phi_r(y))^2k(x,y)\ dxdy+\int_{A_{2r}}u^2(x)\phi_r^2(x)\kappa_{k,A_{4r}}(x)\ dx\\
&\leq \int_{A_{4r}}\int_{A_{4r}}\Big(u(x)^2(\phi_r(x)-\phi_r(y))^2 +(u(x)-u(y))^2\phi_r^2(y)^2\Big)k(x,y)\ dxdy\\
&\qquad\qquad +\int_{A_{2r}}u^2(x)\kappa_{k,A_{4r}}(x)\ dx\\
&\leq \frac{C}{r^2}\int_{A_{4r}}u(x)^2\int_{B_r(x)}|x-y|^2k(x,y)\ dydx+\int_{A_{4r}}u(x)^2\int_{\R^N\setminus B_r(x)}k(x,y)\ dydx\\
&\qquad\qquad + \int_{A_{4r}}\int_{A_{4r}}(u(x)-u(y))^2k(x,y)\ dxdy+\int_{A_{2r}}u^2(x)\kappa_{k,A_{4r}}(x)\ dx\\
&\leq \frac{C}{r^{\epsilon}}\int_{A_{4r}}u(x)^2\int_{B_1(x)}|x-y|^{\epsilon}k(x,y)\ dydx+C\int_{A_{4r}}u(x)^2 dx +b_{k,A_{4r}}(u)+\int_{A_{2r}}u^2(x)\kappa_{k,A_{4r}}(x)\ dx\\
&\leq \frac{C}{r^{\epsilon}}\int_{A_{4r}}u(x)^2\ dx+C\int_{A_{4r}}u(x)^2 dx +b_{k,A_{4r}}(u)+\int_{A_{2r}}u^2(x)\kappa_{k,A_{4r}}(x)\ dx.
\end{align*}
Note here, since $u\in D^k(\R^N)$, we have $\int_{A_{4r}}u(x)^2 dx +b_{k,A_{4r}}(u)\to 0$ for $r\to 0$. Moreover, we have by Lebesgue's differentiation theorem
\begin{align*}
\frac{C}{r^{\epsilon}}\int_{A_{4r}}u(x)^2 dx&\leq \frac{C|B_{4r}|}{r^{\epsilon}}\int_{\partial\Omega}\frac{1}{|B_{4r}|}\int_{B_{4r}(\theta)}u(x)^2\ dx\sigma(d\theta)\\
&\leq Cr^{1-\epsilon}\int_{\partial\Omega}\frac{1}{|B_{4r}|}\int_{B_{4r}(\theta)}u(x)^2\ dx\sigma(d\theta)\to 0\quad\text{for $r\to 0^+$.}
\end{align*}
 Finally, since
\begin{equation}
\label{special case density}
K_{k,\Omega}(u)=\int_{\Omega}u^2(x)\kappa_{k,\Omega}(x)\ dx<\infty
\end{equation}
and, by Lemma \ref{consequences of assumption2}, we have
\[
\kappa_{k,A_{4r}}(x)\leq \int_{\R^N\setminus \Omega}k(x,y)\ dy+\int_{\Omega\setminus A_{4r}}k(x,y)\ dy\leq C \kappa_{k,\Omega}(x)+Cr^{-\epsilon}
\]
for $x\in A_{2r}$, so that also $\int_{A_{2r}}u^2(x)\kappa_{k,A_{4r}}(x)\ dx\to 0$ for $r\to 0$ with a similar argument.
\end{proof}

\begin{proof}[Proof of Theorem \ref{main-thm-density} for $X(\Omega)=\cD^k(\Omega)$]
This statement now follows from Theorem \ref{density2}, Lemma \ref{density-pre1}, and Proposition \ref{density1}.
\end{proof}

\begin{thm}
\label{density3}
Let $\Omega\subset \R^N$ be an open bounded set with Lipschitz boundary. Then $C^{\infty}_c(\Omega)$ is dense in $D^k(\Omega)$. Moreover, if $u\in D^k(\Omega)$ is nonnegative, then we have
\begin{enumerate}
\item There exists a sequence $(u_n)_n\subset D^k(\Omega)\cap L^{\infty}(\Omega)$ with $\lim\limits_{n\to\infty}u_n= u$ in $D^k(\Omega)$ satisfying that for every $n\in \N$ there is $\Omega'_n\subset\subset\Omega$ with $u_n=0$ on $\Omega\setminus \Omega'_n$ and $0\leq u_n\leq u_{n+1}\leq u$.
\item There exists a sequence $(u_n)_n\subset C^{\infty}_c(\Omega)$ with $u_n\geq0$ for every $n\in \N$ and $\lim\limits_{n\to\infty}u_n= u$ in $D^k(\Omega)$.
\end{enumerate}
\end{thm}
\begin{proof}
Consider the Lipschitz map 
$$
g_n:\R\to\R,\qquad g_n(t)=\left\{\begin{aligned}&0 &&t\leq 0\\
& t  && 0<t<n\\
&n && t\geq n.\end{aligned}\right.
$$
Then $v_n:=g_n(u)\in D^k(\Omega)\cap L^{\infty}(\Omega)$ and we have with $\phi_r$ as in the proof of Proposition \ref{density2}
\[
b_{k,\Omega}(u-(1-\phi_r)v_n)\leq b_{k,\Omega}(u-v_n)+b_{k,\Omega}( \phi_rv_n).
\]
Clearly, $b_{k,\Omega}(u-v_n)\to 0$ for $n\to\infty$ by dominated convergence and $b_{k,\Omega}( \phi_rv_n)\to 0$ for $r\to0$ analogously to the proof of Proposition \ref{density2}, noting that the term in \eqref{special case density} reads in this case
$$
K_{k,\Omega}(v_n)\leq n^2\int_{\Omega}\kappa_{k,\Omega}(x)\ dx<\infty \quad\text{for every $n\in \N$.}
$$
In particular, statement (1) follows. Now statement (2) and the density statement follow analogously, again, to the proof of Proposition \ref{density2}.
\end{proof}

\begin{proof}[Proof of Theorem \ref{main-thm-density} for $X(\Omega)=D^k(\Omega)$]
	This statement now follows from Theorem \ref{density3}, Lemma \ref{density-pre1}, and Proposition \ref{density1}.
\end{proof}

\begin{rem}
\label{hardy-bem}
It is tempting to conjecture the following type of Hardy inequality: There is $C>0$ such that 
\[
K_{k,\Omega}(\phi)\leq C\Big(\|\phi\|_{L^2(\Omega)}^2+b_{k,\Omega}(\phi)\Big)\quad\text{for all $\phi\in C^{\infty}_c(\Omega)$}
\]
if $\Omega$ is a bounded Lipschitz set and $k$ is such that its symmetric lower bound $j$ is not in $L^1(\R^N)$. Let us mention that for $k(x,y)=|x-y|^{-2s-N}$ this holds for $s\in(0,1)$, $s\neq \frac{1}{2}$, see \cite{CS03,D04}. Moreover, for $k(x,y)=1_{B_{1}(0)}(x-y)|x-y|^{-N}$, this has been shown in \cite{CW18}. In the general framework presented here, however, it is not clear if this is true.
\end{rem}

\begin{rem}
	\label{rem-weak-defi}
	With the above density results, we can now note that our definition of weak supersolutions (and similarly of weak subsolutions and solutions), see Definition \ref{defi-weak}, can be extended slightly:\\
	Let $u\in \cV^k_{loc}(\Omega)$ satisfy weakly $I_ku\geq f$ in $\Omega$ for some $f\in L^1_{loc}(\Omega)$ and $\Omega\subset \R^N$ open and bounded with Lipschitz boundary. 
	\begin{enumerate}
		\item If $f\in L^2_{loc}(\Omega)$, then by density it also holds
		\begin{equation}
		\label{weak-form2}
		b_k(u,v)\geq \int_{U}f(x)v(x)\ dx\quad\text{for all nonnegative $v\in \cD^k(U)$, $U\subset\subset\Omega$.}
		\end{equation}
		\item If $u\in \cV^k(\Omega)\cap L^{\infty}(\R^N)$ and $f\in L^2(\Omega)$, then by density it also holds
		\begin{equation}
		\label{weak-form3}
		b_k(u,v)\geq \int_{\Omega}f(x)v(x)\ dx\quad\text{for all nonnegative $v\in \cD^k(\Omega)$.}
		\end{equation}
	\end{enumerate}
Finally note that if $u:\R^N\to\R$ satisfies $u1_{U}\in D^k(U)$ for some $U\subset\subset\R^N$ and $u\in L^{\infty}(\R^N\setminus U)$, then $u\in \cV^k_{loc}(U)$.
\end{rem}

\begin{prop}[Weak maximum principle]
	\label{wmp-b}
Assume that the symmetric lower bound $j$ of $k$ defined in \eqref{defi-j} satisfies
	\begin{equation}
	\label{assumption2b}
	\text{$j$ does not vanish identically on $B_r(0)$ for any $r>0$.}
    \end{equation} 
    Let $\Omega\subset \R^N$ open and with Lipschitz boundary, $c\in L^{\infty}_{loc}(\Omega)$, and assume either
	\begin{enumerate}
		\item $c\leq 0$ or 
		\item $\Omega$ and $c$ are such that $\|c^{+}\|_{L^{\infty}(\Omega)}<\inf_{x\in \Omega} \int_{\R^N\setminus \Omega}k(x,y)\ dy$.
	\end{enumerate}
	If $u\in \cV^k(\Omega)$ satisfies in weak sense 
	$$
	I_ku\geq c(x)u\quad\text{ in $\Omega$, $u\geq 0$ almost everywhere in $\R^{N}\setminus \Omega$, and}\quad \liminf\limits_{|x|\to\infty}u(x)\geq 0,
	$$
	then $u\geq 0$ almost everywhere in $\R^N$.
\end{prop}

\begin{proof}
Note that also $u^-\in \cV^k(\Omega)$ and in particular $u^-\in D^k(\Omega)$. Hence, we can find $(v_n)_n\subset C^{\infty}_c(\Omega)$ with $v_n\to u^-$ in $D^k(\Omega)$ for $n\to \infty$ with $0\leq v_n\leq v_{n+1}\leq u^-$ by Proposition \ref{density3}. Then
\begin{align*}
b_{k,\R^N}(u,v_n)\geq \int_{\Omega} c(x) u(x)v_n(x)\ dx\geq -\|c^+\|_{L^{\infty}(\Omega)}\int_{\Omega} u^-(x)v_n(x)\ dx.
\end{align*}
On the other hand, since $u^+v_n=0$ for all $n\in \N$ and $u\geq 0$ almost everywhere in $\R^{N}\setminus \Omega$, we find
\begin{align*}
b_{k,\R^N}(u,v_n)&=b_{k,\Omega}(u,v_n)+\int_{\Omega}v_n(x)\int_{\R^N\setminus \Omega}(u(x)-u(y))k(x,y)\ dydx\\
&\leq b_{k,\Omega}(u^+,v_n)-b_{k,\Omega}(u^-,v_n)+\int_{\Omega}v_n(x)u(x)\int_{\R^N\setminus \Omega} k(x,y)\ dydx\\
&\leq - b_{k,\Omega}(u^-,v_n)-K_{k,\Omega}(u^-,v_n).
\end{align*}
Hence
$$
0\leq \int_{\Omega} u^-(x)v_n(x)\Big(\|c^+\|_{L^{\infty}(\Omega)}-\kappa_{k,\Omega}(x)\Big)\ dx-b_{k,\Omega}(u^-,v_n)\leq -b_{k,\Omega}(u^-,v_n).
$$
Since $v_n\to u^-$ in $D^k(\Omega)$, it follows that $b_{k,\Omega}(u^-,u^-)=0$, but then $u^-$ is constant by Proposition \ref{constant} in $\Omega$. Assume by contradiction that $u^-=m>0$. Then the above calculation gives
\begin{equation}
\label{contradiction}
0\leq m \int_{\Omega}  v_n(x)\Big(\|c^+\|_{L^{\infty}(\Omega)}-\kappa_{k,\Omega}(x)\Big)\ dx,
\end{equation}
which is in both cases a contradiction: If in case 1. $c\leq 0$, then since $\kappa_{k,\Omega}(x)\not\equiv 0$ and since $v_n\to m$ in $D^k(\Omega)$ the right-hand side of \eqref{contradiction} is negative.\\ In case 2. this contradiction is immediate in a similar way.
\end{proof}

\begin{proof}[Proof of Proposition \ref{wmp}]
The statement follows immediately from Proposition \ref{wmp-b}.
\end{proof}

\begin{rem}\label{rem-hardy-wmp}
Usually, the weak maximum principle is stated with an assumption on the first eigenvalue $\Lambda_1(\Omega)$ in place of $\inf_{x\in \Omega}\kappa_{k,\Omega}(x)$. This can be done once the Hardy inequality in Remark \ref{hardy-bem} is shown. Indeed, following the proof of Proposition \ref{wmp-b} gives
$$
-\|c^+\|_{L^{\infty}(\Omega)}\int_{\Omega} u^-(x)v_n(x)\ dx\leq b_{k,\R^N}(u,v_n)\leq - b_{k,\Omega}(u^-,v_n)-K_{k,\Omega}(u^-,v_n)=-b_{k,\R^N}(u^-,v_n).
$$
With $n\to\infty$ and using that by the Hardy inequality it holds $\cD^k(\Omega)=D^k(\Omega)$, it follows that
$$
-\Lambda_1(\Omega)\|u^-\|_{L^2(\Omega)}^2 \geq -b_{k,\R^N}(u^-,u^-)\geq -\|c^+\|_{L^{\infty}(\Omega)}\|u^-\|_{L^2(\Omega)}^2
$$
and the conclusion follows similarly.
\end{rem}

\begin{proof}[Proof of Corollary \ref{wmp-special}]
This statement now follows from Remark \ref{rem-hardy-wmp} using Remark \ref{hardy-bem}.
\end{proof}

\begin{prop}[Strong maximum principle]
	\label{smp-b}
	Assume $k$ satisfies additionally \eqref{assumption2}. Let $\Omega\subset \R^N$ open and $c\in L^{\infty}_{loc}(\Omega)$ with $\|c^+\|_{L^{\infty}(\Omega)}<\infty$. Moreover, let $u\in \cV^k(\Omega)$, $u\geq 0$ satisfy in weak sense $I_ku\geq c(x) u$ in $\Omega$. 
	\begin{enumerate}
		\item If $\Omega$ is connected, then either $u\equiv 0$ in $\Omega$ or $\essinf_Ku>0$ for any $K\subset\subset \Omega$.
		\item If  $j$ given in \eqref{defi-j} satisfies $\essinf_{B_r(0)}j>0$ for any $r>0$, then either $u\equiv 0$ in $\R^N$ or $\essinf_Ku>0$ for any $K\subset\subset \Omega$.
	\end{enumerate}
\end{prop}
\begin{proof}
This statement follows by approximation from \cite[Theorem 2.5 and 2.6]{JW19}. Here, the statement $j\notin L^1(\R^N)$ comes into play since we need
\[
\inf_{x\in B_r(x_0)}\kappa_{k,B_r(x_0)}(x)\to \infty \quad\text{for $r\to 0$}
\]
to conclude the statement for arbitrary $c$ as stated.
\end{proof}

\begin{proof}[Proof of Proposition \ref{smp}]
The statement follows immediately from Proposition \ref{smp-b}.
\end{proof}

\section{On Boundedness}\label{sec:bdd}

In the following, let $h\ast u(x)=\int_{\R^N} h(x-y)u(y)\ dy$ as usual denote the convolution of two functions.

\begin{thm}
	\label{thm-bdd}
	Assume $k$ is such that
	\begin{equation}\label{assumption2}
	\text{the symmetric lower bound $j$ defined in \eqref{defi-j} satisfies }\ \int_{\R^N}j(z)\ dz=\infty
	\end{equation}
	and it holds 
\begin{equation}\label{assumption2d}
	\sup_{x\in \R^N}\int_{K\setminus B_{\epsilon}(x)}k(x,y)^2\ dy<\infty\quad\text{for all $K\subset\subset \R^N$ and $\epsilon>0$.}
\end{equation}
	Let $\Omega\subset \R^N$ be an open set. Let $f\in L^{\infty}(\Omega)$, $h\in L^1(\R^N)\cap L^{2}(\R^N)$, and let $u\in \cV^k_{loc}(\Omega)$ satisfy in weak sense 
	$$
	I_ku\leq \lambda u+h\ast u+f \quad\text{in $\Omega$ for some $\lambda>0$.}
	$$
	If $u^+\in L^{\infty}(\R^N\setminus \Omega')$ for some $\Omega'\subset\subset\Omega$, then $u^+\in L^{\infty}(\R^N)$ and there is $C=C(\Omega,\Omega',k,h,\lambda)>0$ such that
	$$
	\|u^+\|_{L^{\infty}(\Omega')}\leq C\Big(\|f\|_{L^{\infty}(\Omega)}+\|u\|_{L^{2}(\Omega')}+\|u^+\|_{L^{\infty}(\R^N\setminus \Omega')}\Big).
	$$
\end{thm}

\begin{proof}
	Let $\Omega_1,\Omega_2,\Omega_3\subset\R^N$ be with Lipschitz boundary and such that 
	$$
	\Omega'\subset\subset\Omega_1\subset\subset\Omega_2\subset\subset\Omega_3\subset\subset \Omega.
	$$
	Let $\eta\in C^{0,1}_c(\Omega_3)$ such that $0\leq \eta\leq 1$ and $\eta=1$ on $\Omega_2$. Put $v=\eta u$ and, for $\delta>0$, denote $J_{\delta}(x,y):=1_{B_{\delta}(0)}(x-y)k(x,y)$ and $k_{\delta}(x,y)=k(x,y)-J_{\delta}(x,y)$. Note that by Assumption \eqref{assumptions} it follows that $y\mapsto k_{\delta}(x,y)\in L^1(\R^N)$ for all $x\in \R^N$. Moreover, by Assumption \eqref{assumption2}
	$$
	c_{\delta}:=\inf_{x\in \R^N}\int_{\R^N}k_{\delta}(x,y)\ dy\geq \int_{\R^N\setminus B_{\delta}(0)}j(z)\ dz\to \infty \quad\text{for $\delta\to 0$.}
	$$
		Hence, we may fix $\delta>0$ such that
		$$
		c_{\delta}>\lambda.
		$$
		In the following, $C_i>0$, $i=1,\ldots$ denote constants depending on $\Omega'$, $\Omega_i$, $\lambda$, $\delta$, $\Omega$, $\eta$, $k$, and $h$ but may vary from line to line ---clearly, by the choices of these dependencies are actually only through $\lambda$, $\Omega$, $\Omega'$, $\eta$, $k$, and $h$. First note that by Lemma \ref{localization} we have in weak sense
\begin{align*}
	I_kv\leq \lambda u+h\ast u+\tilde{f}\quad\text{in $\Omega_1$}\quad\text{with}\quad \tilde{f}(x)=f(x)+\int_{\R^N\setminus \Omega_2}(1-\eta(y))u(y)k(x,y)\ dy.
	\end{align*}
	In the following, put 
	$$
	A:=\|f\|_{L^{\infty}(\Omega)}+\|u\|_{L^{2}(\Omega')}+\|u^+\|_{L^{\infty}(\R^N\setminus\Omega')}.
	$$
	Then note that for $x\in \R^N$ we have
	$$
|h\ast u(x)|\leq \|h\|_{L^{2}(\R^N)}\|u\|_{L^{2}(\Omega')}+\|u^+\|_{L^{\infty}(\R^N\setminus\Omega')}\|h\|_{L^1(\R^N)}\leq C_1A
	$$
	and, since $\sup\limits_{x\in \Omega_1}\int_{\R^N\setminus \Omega_2}(1-\eta(y))k(x,y)\ dy\leq C_2\sup\limits_{x\in \Omega_1}\int_{\R^N\setminus \Omega_2}\min\{1,|x-y|^{\sigma}\}k(x,y)\ dy<\infty$, it also holds that
	$$
	\|\tilde{f}\|_{L^{\infty}(\Omega_1)}\leq \|f\|_{L^{\infty}(\Omega)}+\|u^+\|_{L^{\infty}(\R^N\setminus \Omega_1)}C_2\leq C_3A.
	$$
	Hence, since $u=v$ in $\Omega_1$, we have in weak sense
	\begin{align*}
	I_kv\leq \lambda v+C_4A\quad\text{in $\Omega_1$.}
	\end{align*}
	Next, let $\mu\in C^{\infty}_c(\Omega'')$ for some $\Omega'\subset\subset\Omega''\subset\subset\Omega_1$ such that $0\leq \mu\leq 1$, $\mu=1$ on $\Omega'$, and $\mu=0$ on $\R^N\setminus \Omega''$. Let $\phi_t=\mu^2(v-t)^+\in \cD^k(\Omega'')$ for $t>0$ and note that
	\begin{equation}
	\label{bdd-1}
	b_k(v,\phi_t)\leq\int_{\Omega''}(\lambda v(x)+C_4A)\phi_t(x)\ dx.
	\end{equation} 
 Fix $t>0$ such that
	\begin{align*}
	 t\geq \|u^+\|_{L^{\infty}(\R^N\setminus \Omega')}&\quad\text{and}\quad C_6A+(\lambda-c_{\delta})t\leq 0,\quad\text{where}\\
	C_6&=C_4+C_5\quad\text{with}\quad C_5=\sup_{x\in \Omega'}\left(\int_{\Omega'}k_{\delta}^2(x,y)\ dy\right)^{1/2}+\sup_{x\in \Omega'}\int_{\R^N\setminus\Omega'}k_{\delta}(x,y)\ dy.
	\end{align*}
	That is, we fix 
	$$
	t= A\Big(1+\frac{C_6}{c_{\delta}-\lambda}\Big).
	$$
	Then with \eqref{bdd-1}
	\begin{align*}
	&b_{J_{\delta}}(v,\phi_t)=b_k(v,\phi_t)-b_{k_{\delta}}(v,\phi_t)\\
	&\leq \int_{\Omega''}\lambda v(x)\phi_t(x)+C_4A\phi_t(x)\ dx-\int_{\R^N}v(x)\phi_t(x)\int_{\R^N}k_{\delta}(x,y)\ dydx+\int_{\R^N}\phi_t(x)\int_{\R^N}v(y)k_{\delta}(x,y)\ dy\ dx.
	\end{align*}
	Note here, that for $x\in \R^N$ we have by the integrability assumptions on $k_{\delta}$ and $k$
	\begin{align*}
	\int_{\R^N}v(y)k_{\delta}(x,y)\ dy&\leq \int_{\R^N}u(y)k_{\delta}(x,y)\ dy\leq C_5(\|u\|_{L^2(\Omega')}+\|u^+\|_{L^{\infty}(\R^N\setminus \Omega')})\leq C_5A
		\end{align*}
		so that using that $v\geq t$ in $\supp\,\phi_t$ we have
\begin{equation}
	\label{bdd-2}
	b_{J_{\delta}}(v,\phi_t)\leq  \int_{\Omega''}(C_6A+(\lambda-c_{\delta})v(x))\phi_t(x)\ dx\leq  (C_6A+(\lambda-c_{\delta})t)\int_{\Omega''}\phi_t(x)\ dx.
	\end{equation}
	On the other hand, with $v_t(x)=v(x)-t$, we have
	\begin{align*}
	(v(x)-v(y))&(\phi_t(x)-\phi_t(y))-(\mu(x)v_t^+(x)-\mu(y)v_t^+(y))^2\\
	&=2\mu(x)\mu(y)v_t^+(x)v_t^+(y)-v_t(y)\mu^2(x)v_t^+(x)-\mu^2(y)v_t^+(y)v_t(x)\\
	&=-v_t^+(x)v_t^+(y)(\mu(x)-\mu(y))^2+v_t^-(y)\mu^2(x)v_t^+(x)+\mu^2(y)v_t^+(y)v_t^-(x)\\
	&\geq -v_t^+(x)v_t^+(y)(\mu(x)-\mu(y))^2.
	\end{align*}
	Hence with Poincar\'e's inequality, using that by Assumption \ref{assumption2} there is for any $K\subset \R^N$ open and bounded, some $C>0$ such that $b_{J_{\delta}}(u)\geq C\|u\|_{L^2(\R^N)}^2$ for $u\in \cD^{J_{\delta}}(K)$, we find for some constant $C_7$
	\begin{align}
	b_{J_{\delta}}(v,\phi_t)&\geq b_{J_{\delta}}(\mu v_t^+)-\frac{1}{2}\int_{\R^N}\int_{\R^N} v_t^+(x)v_t^+(y)(\mu(x)-\mu(y))^2J_{\delta}(x,y)\ dxdy\notag\\
	&\geq C_7\int_{\R^N}\mu^2(x)(v_t^+(x))^2\ dx-\frac{1}{2}\int_{\R^N}\int_{\R^N} v_t^+(x)v_t^+(y)(\mu(x)-\mu(y))^2J_{\delta}(x,y)\ dxdy\\
	&= C_7\int_{\R^N}\mu^2(x)(v_t^+(x))^2\ dx-\frac{1}{2}\int_{\Omega'}\int_{\Omega'} v_t^+(x)v_t^+(y)(\mu(x)-\mu(y))^2J_{\delta}(x,y)\ dxdy\\
	&= C_7\int_{\R^N}\mu^2(x)(v_t^+(x))^2\ dx\geq C_7\int_{\Omega'}(v_t^+(x))^2\ dx.\label{bdd-3}
	\end{align}
	Combining \eqref{bdd-3} and \eqref{bdd-2} we have
	\begin{align*}
	C_7\int_{\Omega'}&(v_t^+(x))^2\ dx\leq \Big(C_6A+(\lambda-c_{\delta})t\Big)\int_{\Omega''} \phi_t(x)\ dx\leq 0.
	\end{align*}
	Hence $v_t^+=0$ in $\Omega'$ and thus $u=v\leq t=A\cdot C_{8}$ in $\Omega'$ as claimed. 
\end{proof}

\begin{cor}
	\label{thm-bdd2}
	If in the situation of Theorem \ref{thm-bdd} we have in weak sense $I_ku=\lambda u+h\ast u+f$ in $\Omega$, then we have $u\in L^{\infty}(\Omega')$ and there is $C=C(\Omega,\Omega',k,\lambda,h)>0$ such that
	$$
	\|u\|_{L^{\infty}(\Omega')}\leq C\Big(\|f\|_{L^{\infty}(\Omega)}+\|u\|_{L^{2}(\Omega')}+\|u\|_{L^{\infty}(\R^N\setminus \Omega')}\Big).
	$$
\end{cor}

\begin{proof}
	This follows by replacing $u$ with $-u$ (and $f$ with $-f$) in the statement of Theorem \ref{thm-bdd}.
\end{proof}

\begin{thm}\label{thm-bdd-full-domain}
	If in the situation of Theorem \ref{thm-bdd} we have in weak sense $I_ku=\lambda u+h\ast u+f$ in $\Omega$ and $u\in \cD^k(\Omega)$, then we have $u\in L^{\infty}(\Omega)$ and there is $C=C(\Omega,k,\lambda,h)>0$ such that
	$$
	\|u\|_{L^{\infty}(\Omega)}\leq C\Big(\|f\|_{L^{\infty}(\Omega)}+\|u\|_{L^{2}(\Omega)}\Big).
	$$
\end{thm}

\begin{proof}
	Using in the proof of Theorem \ref{thm-bdd} the test-function $u_t^+$ instead of $\phi_t$ (and similarly for Corollary \ref{thm-bdd2}), we find
	$$
	\|u\|_{L^{\infty}(\Omega)}\leq C\Big(\|f\|_{L^{\infty}(\Omega)}+\|u\|_{L^{2}(\Omega)}\Big).
	$$
	as claimed.
\end{proof}

\section{On differentiability of solutions}\label{sec:reg}

In the following, $\Omega\subset \R^N$ is an open bounded set and $k$ satisfies through out the assumptions \eqref{assumptions} with some $\sigma<\frac12$, \eqref{assumption2}, and \eqref{assumption2d}. Moreover, we assume that there is $j:\R^N\to[0,\infty]$ such that $k(x,y)=j(x-y)$ for $x,y\in \R^N$ and that for some $m\in \N\cup\{\infty\}$ the following holds: We have $j\in W^{l,1}(\R^N\setminus B_{\epsilon}(0)$ for every $l\in \N$ with $l\leq 2m$, and there is some constant $C_j>0$ such that
$$
|\nabla j(z)|\leq C_j|z|^{-1-\sigma-N}\quad \text{for all $0<|z|\leq 3$.}
$$ 
For simplicity, we write $j$ in place of $k$ and fix 
$$
\alpha:=1-\sigma\in(\frac12,1).
$$

\begin{thm}
	\label{thm:differentiability-step1}
	Let $f\in H^1(\Omega)$, $\lambda\in \R$ and $u\in \cV^j_{loc}(\Omega)\cap L^{\infty}(\R^N)$ satisfy in weak sense $I_ju=f+\lambda u$ in $\Omega$. Then for any $\Omega'\subset\subset \Omega$ there is $C=C(N,\Omega,\Omega',j,\lambda)>0$ such that
	\begin{equation}
		\label{bound-derivative}
		\| \delta_{h,e} u\|_{L^2(\Omega')}\leq h^{\alpha}C\Big(\|f\|_{H^1(\Omega)}^2+\|u\|_{L^\infty(\R^N)}^2\Big)^{\frac12}\quad\text{for all $h>0$, $e\in \partial B_1(0)$.}
	\end{equation}
\end{thm}

\begin{proof}
	Let $\Omega'\subset\subset\Omega$ and fix $r\in(0,\frac18)$ small such that $8r\leq\dist(\Omega',\R^N\setminus\Omega)$. Moreover, fix $x_0\in \Omega'$ and denote $B_n:=B_{nr}(x_0)$. Note that by using assumption \eqref{assumption2} with Lemma \ref{poincare} we achieve, by making $r>0$ small enough,
	$$
	\lambda<\lambda_1=\min_{\substack{w\in \cD^j(B_4)\\ w\neq 0}}\frac{b_j(w)}{\|w\|_{L^2(\R^N)^2}}.
	$$
	Let $\eta\in C^{0,1}_c(B_4)$ with $0\leq \eta\leq 1$, $\eta\equiv 1$ on $B_2$. Note that it holds
	$$
	|\eta(x)-\eta(y)|\leq 2\|\eta\|_{C^{0,1}(\R^N)}\min\{1,|x-y|\},
	$$
	where we put as usual
	$$
	\|\eta\|_{C^{0,1}(\R^N)}:=\sup_{x\in \R^N}|\eta(x)|+\sup_{\substack{x,y\in \R^N\\ x\neq y}}\frac{|\eta(x)-\eta(y)|}{|x-y|}.
	$$
	Note that by choice we have $\|\eta\|_{C^{0,1}(\R^N)}\leq 1+\frac{1}{r}\leq \frac{2}{r}$, so that for all $x,y\in \R^N$
	\begin{equation}
		\label{bound on eta}
		|\eta(x)-\eta(y)|\leq \frac{4}{r}\min\{1,|x-y|\}.
	\end{equation}
	Fix $e\in \partial B_1(0)$ and $h\in(0,r)$. Let 
	$$
	A:=\|u\|_{L^{\infty}(\R^N)}.
	$$
	Let $\psi=\eta^2\delta_{h}u\in \cD^j(B_4)$, where in the following $\delta_hu:=\delta_{h,e}u$. Note that
	\begin{align*}
		(\delta_{h}u(x)-\delta_{h}u(y))(\psi(x)-\psi(y))&=(\eta(x)\delta_{h}u(x)-\eta(y)\delta_{h}u(y))^2\\
		&\qquad -\delta_{h}u(x)\partial_{h}u(y)(\eta(x)-\eta(y))^2.
	\end{align*}
	Hence, we have
	$$
	b_j(\delta_{h}u,\psi)=b_j(\eta\delta_{h}u)-\frac{1}{2}\iint_{\R^{N}\times\R^N}\delta_{h}u(x)\delta_{h}u(y)(\eta(x)-\eta(y))^2j(x-y)\ dxdy.
	$$
	and using the translation invariance, we also have
	$$
	b_j(\delta_h u,\psi)=\int_{\Omega} [\delta_hf(x)+\lambda\delta_hu]\psi(x)\ dx.
	$$
	In the following, for simplicity, we put $v(x)=\eta(x)\delta_hu(x)$, $x\in \R^N$. Note that by Definition, $v\in \cD^j(B_4)$. Then with the help of Young's inequality for some $\mu\in(0,1)$ such that
	\begin{equation}
		\label{condition on mu}
		2\mu< \lambda_1-\lambda
	\end{equation}
	we find
	\begin{align}
		&\lambda_1\|v\|_{L^2(\Omega'')}^2\leq b_j(v)=b_j(\delta_{h}u,\psi)+\frac{1}{2}\iint_{\R^{n}\times\R^N}\delta_{h}u(x)\delta_{h}u(y)(\eta(x)-\eta(y))^2j(x-y)\ dxdy\notag\\
		&=\int_{\Omega''}[\delta_{h}f(x)+\lambda\delta_hu]\eta^2(x)\delta_{h}u(x)\ dx+\frac{1}{2}\iint_{\R^{n}\times\R^N}\delta_{h}u(x)\delta_{h}u(y)(\eta(x)-\eta(y))^2j(x-y)\ dxdy\notag\\
		&\leq  (\mu+\lambda)\|v\|_{L^{2}(\Omega'')}^2+\mu^{-1}h^2\|\frac{\delta_hf}{h}\|_{L^2(\Omega'')}^2  +\frac{1}{2}\iint_{\R^N\times\R^N}\delta_{h}u(x)\delta_{h}u(y)(\eta(x)-\eta(y))^2j(x-y)\ dxdy.\label{basis-bound}
	\end{align}
	By a rearrangement of the double integral with Young's inequality for the same $\mu\in(0,1)$ as above 
	we have
	\begin{align}
		\frac{1}{2}&\iint_{\R^{n}\times\R^N}\delta_{h}u(x)\delta_{h}u(y)(\eta(x)-\eta(y))^2j(x-y)\ dxdy\notag\\
		&=\iint_{\R^N\times \R^N} \delta_{h}u(x)\delta_{h}u(y)\eta(x)(\eta(x)-\eta(y))j(x-y)\ dxdy\notag\\
		&=\int_{\R^N} \eta(x)\delta_{h}u(x)\int_{\R^N}u(y)\delta_{-h,y}\Big((\eta(x)-\eta(y))j(x-y)\Big)\ dydx\notag\\
		&\leq \mu\| v\|_{L^2(B_4)}^2+\mu^{-1}\int_{B_4}\Bigg(\ \int_{\R^N}|u(y)|\Big|\delta_{-h,y}\Big((\eta(x)-\eta(y))j(x-y)\Big)\Big|\ dy\Bigg)^2 dx\notag\\
		&\leq \mu\| v\|_{L^2(B_4)}^2+\mu^{-1}A^2\int_{B_4}\Bigg(\ \int_{\R^N}\Big|\delta_{-h,y}\Big((\eta(x)-\eta(y))j(y-x)\Big)\Big|\ dy\Bigg)^2 dx\notag\\
		&\leq 
		\mu\| v\|_{L^2(B_4)}^2+\mu^{-1}A^2\int_{B_4}\Bigg(\ \int_{\R^N}\Big|\delta_{-h,z}\Big((\eta(x)-\eta(z+x))j(z)\Big)\Big|\ dz\Bigg)^2 dx.
		\label{bounding-1}
	\end{align}
	Here, we indicate with $\delta_{-h,y}$ (resp. $\delta_{-h,z}$) that $\delta_{-h}$ acts on the $y$ (resp. $z$) variable. Note that
	\begin{align}
		\delta_{-h,z}&\Big((\eta(x)-\eta(z+x))j(z)\Big)\notag\\
		&=\delta_{-h,z}(\eta(x)-\eta(z+x))j(z)+ (\eta(x)-\eta(z+x-he))\delta_{-h}j(z)\notag\\
		&=\Big(\eta(z+x)-\eta(z+x-he)\Big) j(z)+(\eta(x)-\eta(z+x-he))\Big(j(z-he)-j(z)\Big)\label{difference-op1}\\
		&= \Big(\eta(z+x)-\eta(z)\Big) j(z)+ \Big(\eta(x)-\eta(z+x-he)\Big)j(z-he).\label{difference-op2}
	\end{align}
	
	Note here, that \eqref{difference-op1} satisfies
	\begin{equation}
		\label{difference-op1b}
		\begin{split}
			\Bigg|\Big(\eta(z+x)-\eta(z+x-he)\Big)& j(z)+(\eta(x)-\eta(z+x-he))\Big(j(z-he)-j(z)\Big)\Bigg|\\
			&\leq \frac{4h}{r}J(z)+\frac{4h}{r}\min\{1,|z-he|\}\int_0^1|\nabla j(z-\tau he)|\ d\tau
		\end{split}
	\end{equation}
	and \eqref{difference-op2} can be written as
	\begin{equation}
		\label{difference-op2b}
		\begin{split}
			\Bigg| \Big(\eta(z+x)-\eta(z)\Big) &j(z)+ \Big(\eta(x)-\eta(z+x-he)\Big)j(z-he)\Bigg|\\
			&\leq \frac{4}{r}\min\{1,|z|\}j(z)+\frac{4}{r}\min\{1,|z-he|\}j(z-he).
		\end{split}
	\end{equation}
	For $h\in(0,r)$, $z\in \R^N\setminus\{0\}$ put
	\begin{align*}
		k_h(z)=\min&\Bigg\{h\Big(j(z)+\min\{1,|z-he|\}\int_0^1|\nabla j(z-\tau he)|\ d\tau\Big),\\
		&\qquad\qquad \min\{1,|z|\}j(z)+\min\{1,|z-he|\}j(z-he)\Bigg\}.
	\end{align*}
	Then, by combining \eqref{basis-bound} and \eqref{bounding-1}, we find
	\begin{equation}
		\label{final-bound1}
		\begin{split}
			\|\delta_h u\|_{L^2(B_2)}^2&\leq \|v\|_{L^2(B_4)}^2\\
			&\leq \frac{\mu^{-1}|B_4|}{\lambda_1-\lambda-2\mu}\Bigg(h^2\|\frac{\delta_hf}{h}\|_{L^{2}(B_4)}^2+ \frac{16}{r^2}\|u\|_{L^\infty(\R^N)}^2 \Big(\ \int_{\R^N}k_h(z)\ dz\Big)^2\Bigg).
		\end{split}
	\end{equation}
	Next we show that we have $\int_{\R^N}k_h(z)\ dz \leq C h^{\alpha}$ for some $C>0$. Clearly, we can bound
	\begin{equation}
		\label{k-h-bound outside}
		\int_{\R^N\setminus B_2(0)}k_h(z)\ dz\leq C_1h
	\end{equation}
	for some $C_1=C_1(N,j)>0$, using that $B_1(0)\cup B_1(he)\subset B_2(0)$ and the properties of $j$. In the following, by making $C_j$ larger if necessary, we may also assume that assumption \eqref{assumptions} reads
	$$
	\sup_{x\in \R^N}\int_{\R^N}\min\{1,|x-y|^{\sigma}\}j(x-y)\ dy=\int_{\R^N}\min\{1,|z|^{\sigma}\}j(z)\ dz\leq C_j.
	$$
	 Then note that $B_{2h}(he)\subset B_{3h}(0)$ and we have
	\begin{align}
		\int_{B_{2h}(0)} &\min\{1,|z|\}j(z)+\min\{1,|z-he|\}j(z-he)\ dz\notag\\
		&\leq C_j\int_{B_{3h}(0)} |z|^{1-\sigma-n}\ dz + C_j\int_{B_{3h}(he)}|z-he|^{1-\sigma-n}\ dz\notag\\
		&=\frac{2|B_{1}(0)|C_j}{n} \int_0^{3h}\rho^{-\sigma}\ d\rho=\frac{2|B_{1}(0)|C_j}{n(1-\sigma)}(3h)^{1-\sigma}.\label{k-h-bound inside1}
	\end{align}
	While with $b_{\sigma}(t)=\frac{1}{\sigma}t^{-\sigma}$ we have
	\begin{align}
		h\int_{B_2(0)\setminus B_{2h}(0)}&j(z)+\min\{1,|z-he|\}\int_0^1|\nabla j(z-\tau he)|\ d\tau\ dz\notag\\
		&\leq \frac{h|B_{1}(0)|C_j}{n}\int_{2h}^2\rho^{-\sigma-1}\ d\rho+hC_j\int_0^1 \int_{B_{3}(\tau he)\setminus B_{h}(\tau he)}|z||z-\tau he|^{-1-\sigma-n}\ dz\ d\tau\notag\\
		&\leq \frac{h|B_{1}(0)|C_j}{n}b_{\sigma}(2h)+hC_j\int_0^1 \int_{B_{3}(0)\setminus B_{h}(0)}|z+\tau h e||z|^{-1-\sigma-n}\ dz\ d\tau\notag\\
		&\leq \frac{h|B_{1}(0)|C_j}{n} b_{\sigma}(2h)+hC_J \int_{B_{3}(0)\setminus B_{h}(0)}|z|^{-\sigma-n}\ dz +h^2C_j \int_{B_{3}(0)\setminus B_{h}(0)}|z|^{-1-\sigma-n}\ dz\notag\\
		&\leq \frac{2h|B_{1}(0)|C_j}{n} b_{\sigma}(h)+\frac{h^2|B_{1}(0)|C_j}{n} \int_h^3\rho^{-2-\sigma}\ d\rho\notag\\
		&\leq  \frac{2|B_{1}(0)|C_j}{n} hb_{\sigma}(h)+\frac{|B_{1}(0)|C_j}{n(1+\sigma)}h^{1-\sigma}.\label{k-h-bound inside2}
	\end{align}
	Combining \eqref{k-h-bound outside} with \eqref{k-h-bound inside1} and \eqref{k-h-bound inside2} and the choice $\alpha=1-\sigma\in(0,1)$ we find $C_2=C_2(N, j,\alpha)>0$ such that
	\begin{equation}
		\label{l1-bound}
		\int_{\R^N}k_h(z)\ dz\leq C_2h^{\alpha}.
	\end{equation}
	Hence, from \eqref{final-bound1} with \eqref{l1-bound} we have
	\begin{equation}
		\label{thm-differentiability-step1-single}
		\|\delta_h u\|_{L^2(B_2)}^2\leq \|v\|_{L^2(B_4)}^2\leq h^{2\alpha}C_4\Big(\|\frac{\delta_hf}{h}\|_{L^2(B_4)}^2+ \|u\|_{L^\infty(\R^N)}^2\Big),
	\end{equation}
	for a constant $C_4=C_4(N,j,r,\alpha,\lambda)>0$. By a standard covering argument, we then also find with a constant $C_5=C_5(N,j,\Omega,\Omega',\alpha,\lambda)>0$ and $\Omega''=\{x\in \Omega\;:\; \dist(x,\R^N\setminus\Omega)>4r\}$
	\begin{equation}
		\label{thm-differentiability-step1}
		\|\delta_h u\|_{L^2(\Omega')}^2\leq h^{2\alpha}C_4\Big(\|\frac{\delta_hf}{h}\|_{L^2(\Omega'')}^2+ \|u\|_{L^\infty(\R^N)}^2\Big),
	\end{equation}
The claim \eqref{bound-derivative} then follows since $f\in H^1(\Omega)$.
\end{proof}

\begin{rem}
	\label{thm-differentiability-0}
	If additionally $f\in L^{\infty}(\Omega)$, combining Theorem \ref{thm:differentiability-step1} with Corollary \ref{thm-bdd2} it follows that we have in the situation of Theorem \ref{thm:differentiability-step1} for every $\Omega'\subset\subset\Omega$
	\begin{equation}
		\label{bound-derivative2}
		\| \delta_{h,e} u\|_{L^2(\Omega')}\leq h^{\alpha}C\Big(\|f\|_{C^1(\Omega)}^2+\|u\|_{L^2(\Omega')}^2+\|u\|_{L^\infty(\R^N\setminus \Omega')}^2\Big)^{\frac12}\quad\text{for all $h>0$, $e\in \partial B_1(0)$.}
	\end{equation}
\end{rem}

\begin{cor}
	\label{thm:differentiability}
	Assume $m=1$. Let $f\in C^2(\overline{\Omega})$, $\lambda\in \R$, and let $u\in \cV^j_{loc}(\Omega)\cap L^{\infty}(\R^N)$ satisfy in weak sense $I_ju=\lambda u+ f$ in $\Omega$. Then $u\in H^{1}(\Omega')$ and $\partial_iu\in D^j(\Omega')$ for any $\Omega'\subset \subset \Omega$. More precisely, with $\alpha$ as above there is for any $\Omega'\subset\subset \Omega$ a constant $C=C(N,\Omega,\Omega', j,\lambda)>0$ such that
	\begin{equation}
		\label{bound-derivative-final1}
		\sup_{\substack{ e\in \partial B_1(0)\\ h>0}}h^{-2\alpha}\| \delta_{h,e}^2 u\|_{L^2(\Omega')}\leq C\Big(\|f\|_{C^2(\Omega)}^2+\|u\|_{L^2(\Omega')}+\|u\|^2_{L^\infty(\R^N\setminus \Omega')} \Big)^{\frac12},
	\end{equation}
	so that $u\in N^{2\alpha,2}(\Omega')\subset H^1(\Omega')$, that is, there is also $C'=C'(N, j,\Omega,\Omega',\alpha,\lambda)>0$ such that
	\begin{equation}
		\label{bound-derivative-final2}
		\| \nabla u\|_{L^2(\Omega')}\leq C'\Big(\|f\|_{C^2(\Omega)}^2+\|u\|_{L^2(\Omega')}+\|u\|^2_{L^\infty(\R^N\setminus \Omega')} \Big)^{\frac12}
	\end{equation}
	and, moreover, 
	$$
	b_{j,\Omega'}(\partial_iu)\leq C'\quad\text{for $i=1,\ldots,N$.}
	$$
\end{cor}
\begin{proof}
	Let $\Omega_i\subset\subset \Omega$, $i=1,\ldots,7$ such that
	$$
	\Omega'\subset\subset \Omega_i\subset\subset \Omega_j \quad\text{for $1\leq i<k\leq 7$.}
	$$ 
	Let $\eta\in C^{\infty}_c(\Omega_{7})$ with $\eta=1$ on $\Omega_6$ and $0\leq \eta\leq 1$. Fix $e\in \partial B_1(0)$ and $h\in(0,\frac12r)$, where $r=\min\{\dist(\Omega_i,\Omega\setminus \Omega_{i+1})\;:\; i=1,\ldots,6\}$. Then by Lemma \ref{localization} the function $v=\eta\delta_h u$, where we write $\delta_h$ instead of $\delta_{h,e}$, satisfies $I_jv=\lambda v+\tilde{f}$ in $\Omega_5$, where $\tilde{f}=\delta_hf+g_{\eta,\delta_hu}$. Following the proof of Theorem \ref{thm:differentiability-step1} to \eqref{thm-differentiability-step1} it follows with Theorem \ref{thm-bdd} that there is $C=C(N, j, r,\alpha,\lambda)>0$ (changing from line to line) such that
	\begin{align*}
		\| \delta_{h}^2 u\|^2_{L^2(\Omega')}&= \|\delta_hv\|^2_{L^2(\Omega')}\leq h^{2\alpha}C\Big(\|\frac{\delta_{h}\tilde{f}}{h}\|_{L^{2}(\Omega_1)}^2+ \|v\|^2_{L^\infty(\R^N)}\Big)\\
		&\leq h^{2\alpha}C\Big(\|\frac{\delta_{h}\tilde{f}}{h}\|_{L^{2}(\Omega_1)}^2+ \|\tilde{f}\|_{L^{\infty}(\Omega_4)}^2+\|v\|_{L^{2}(\Omega_3)}^2+\|v\|_{L^{\infty}(\R^N\setminus \Omega_3)}^2\Big)\\
		&\leq h^{2\alpha}C\Big(\|\frac{\delta_{h}\tilde{f}}{h}\|_{L^{2}(\Omega_1)}^2+ \|\tilde{f}\|_{L^{\infty}(\Omega_4)}^2+\|\delta_hu\|_{L^{2}(\Omega_3)}^2\Big)\\
		&\leq h^{2\alpha}C\Big(\|\frac{\delta_{h}\tilde{f}}{h}\|_{L^{2}(\Omega_1)}^2+ \|\tilde{f}\|_{L^{\infty}(\Omega_4)}^2+h^{2\alpha}\Big(\|f\|_{C^1(\Omega)}^2+\|u\|_{L^{\infty}(\R^N)}^2\Big)\Big),
	\end{align*}
	where we applied once more Theorem \ref{thm:differentiability-step1}. Here, for $x\in \Omega_4$ using the assumptions on the differentiability of $j$ it follows that there is $C=C(j)>0$ such that
	\begin{align*}
		|\tilde{f}(x)|&\leq |\delta_hf(x)|+\Bigg|\int_{\R^N\setminus \Omega_6}(1-\eta(y))\delta_hu(y)j(x-y)\ dy\Bigg|\\
		&= |\delta_hf(x)|+\Bigg|\int_{\R^N\setminus \Omega_5}|u(y)|\delta_h[(1-\eta(y))j(x-y)]\ dy\\
		&\leq hC\Big(\|\nabla f\|_{L^{\infty}(\Omega)} +\|u\|_{L^{\infty}(\R^N\setminus \Omega')}\Big).
	\end{align*}
	Moreover, for $x\in \Omega_1$ in a similar way, there is $C=C(j)>0$ such that
	\begin{align*}
		|\delta_h\tilde{f}(x)|&\leq |\delta_h^2f(x)|+\Bigg|\int_{\R^N\setminus \Omega_6}(1-\eta(y))\delta_hu(y)\delta_hj(x-y)\ dy\Bigg|\\
		&\leq h^2\|f\|_{C^2(\Omega)}+ \|u\|_{L^{\infty}(\R^N\setminus \Omega')}\Bigg|\int_{\R^N\setminus \Omega_5}\delta_h[(1-\eta(y))\delta_hj(x-y)]\ dy\Bigg|\\
		&\leq h^2C\Bigg(\|f\|_{C^2(\Omega)}+ \|u\|_{L^{\infty}(\R^N\setminus \Omega')}\Bigg).
	\end{align*}
	Thus we have
	$$
	\| \delta_{h}^2 u\|^2_{L^2(\Omega')}\leq Ch^{4\alpha}\Bigg(\|f\|_{C^{2}(\Omega)}^2+\|u\|_{L^2(\Omega')}^2+\|u\|_{L^{\infty}(\R^N\setminus \Omega')}^2\Bigg).
	$$
	The proof of the first part then is finished with Proposition \ref{prop-sobolev-nikolskii} since $2\alpha>1$. Next, write $D_hp(x)=\frac{p(x+he)-p(x)}{h}$ for any function $p:\R^N\to\R$, with $e\in \partial B_1(0)$ fixed and $h\in \R\setminus\{0\}$. Then with Lemma \ref{localization} for some $\eta\in C^{\infty}_c(\Omega)$ such that $0\leq \eta\leq1$ and $\eta\equiv 1$ on $\Omega_2\subset\subset \Omega$ with $\Omega'\subset\subset \Omega_1\subset\subset \Omega_2$ we have with $v=\eta u$, 
	$$
	I_jv=f+\lambda v+g_{\eta,u}\quad\text{in $\Omega_1$, where}\quad g_{\eta,u}=\int_{\R^N\setminus\Omega_2}(1-\eta(y))u(y)j(x-y)\ dy.
	$$
	Next, let $\mu\in C^{\infty}_c(\Omega_1)$ with $0\leq \mu\leq1$ and $\mu\equiv 1$ on $\Omega'$. Then with $\phi=D_{-h}[\mu^2D_hv]\in \cD^j(\Omega_1)$ for $h$ small enough we have for some $C>0$ (which may change from line to line independently of $h$)
	\begin{align}
	\label{bdd-rhs-derivative}
	|b_j(v,\phi)|=\Bigg|\int_{\Omega_1}D_hf\mu^2D_hv+ \lambda (\mu D_hv)^2+D_hg_{\eta,u}\mu^2 D_hv\ dx\Bigg|\leq C,
	\end{align}
	since
	$$
	\int_{\Omega_1}|D_hf\mu^2D_hv|\ dx\leq C\|f\|_{C^1(\Omega)}\|\nabla u\|_{L^2(\Omega_2)}<\infty,\qquad \int_{\Omega_1}|\lambda (\mu D_hv)^2|\ dx\leq 2|\lambda|\|\nabla u\|_{L^2(\Omega_2)}^2<\infty,
	$$
	and
$$
\int_{\Omega_1}|D_hg_{\eta,u}\mu^2 D_hv|\ dx\leq C\Big(\int_{\Omega_1}\int_{\R^N\setminus \Omega_2}|(1-\eta(y))u(y)|[D_hj](x-y)|\ dy\ dx\Big)^{1/2}\|\nabla u\|_{L^2(\Omega_2)}<\infty
$$
due to assumptions on the differentiability of $j$. Moreover, with a similar calculation as in the proof of Theorem \ref{thm:differentiability-step1} we have
\begin{align*}
b_j(v,\phi)&=b_j(\mu D_hv,\mu D_hv)-\frac{1}{2}\int_{\R^N}\int_{\R^N} D_hv(x)D_hv(y)(\mu(x)-\mu(y))^2j(x-y)\ dxdy,
\end{align*}
where for some $\Omega_2\subset\subset\Omega_3\subset\subset\Omega_4\subset\subset \Omega$~ with $h$ small enough
\begin{align*}
\int_{\R^N}\int_{\R^N} |D_hv(x)D_hv(y)(\mu(x)-\mu(y))^2&j(x-y)|\ dxdy\\
&\leq C\int_{\Omega_3} \int_{\Omega_3}|D_h(\eta u)(x)D_h(\eta u)(y)||x-y|^2j(x-y)\ dxdy\\
&\leq C\int_{\Omega_3} |D_h(\eta u)(x)|^2\int_{\Omega_3} |x-y|^2j(x-y)\ dydx\\
&\leq C\|\nabla u\|_{L^2(\Omega_4)}\int_{\R^N}\min\{1,|z|^2\}j(z)\ dz<\infty.
\end{align*}
Combining this with \eqref{bdd-rhs-derivative} we find
$$
b_j(\mu D_hv,\mu D_hv)\leq C\quad\text{for all $h>0$ small enough.}
$$
Since also $\mu D_hv\in \cD^j(\Omega_2)$ for all $h>0$ small enough (see Lemma \ref{properties1}) and since $D^j(\Omega_2)$ is a Hilbert space, we conclude that $\mu\partial_ev\in \cD^j(\Omega_2)$ with
$$
b_j(\mu \partial_ev)\leq C
$$
for $h\to 0$. This finishes the proof.
\end{proof}

\begin{cor}
	\label{thm:differentiability3}
	Let $f\in C^{2m}(\overline{\Omega})$, $\lambda\in \R$, and let $u\in \cV^j_{loc}(\Omega)\cap L^{\infty}(\R^N)$ satisfy in weak sense $I_ju=\lambda u+f$ in $\Omega$. Then $u\in H^{m}(\Omega')$ for any $\Omega'\subset \subset \Omega$ and there is $C=C(N, j,\Omega,\Omega',m)>0$ such that
	\begin{equation}
		\label{bound-derivative-final3eq}
		\| u\|_{H^m(\Omega')}\leq C\Big(\|f\|_{C^{2m}(\Omega)}^2+\|u\|^2_{L^2(\Omega')}+\|u\|^2_{L^\infty(\R^N\setminus \Omega')} \Big)^{\frac12}.
	\end{equation}
	In particular, if $m=\infty$, then $u\in C^{\infty}(\Omega)$.
\end{cor}
\begin{proof}
	By Corollary \ref{thm:differentiability} the claim holds for $m=1$ in particular with $u|_{\Omega'}\in D^j(\Omega')$ for all $\Omega'\subset\subset \Omega$. Assume next, the claim holds for $m-1$ with $m\in\N$, $m\geq 2$ in the following way: We have $u\in H^{m-1}(\Omega')$ and $\partial^{\beta}u|_{\Omega'}\in D^j(\Omega')$ for any $\Omega'\subset\subset \Omega$ and $\beta\in \N_0^N$ with $|\beta|\leq m-1$, and there is $C=C(N, j,\Omega,\Omega',m)>0$ such that
	\begin{equation}
		\label{bound-derivative-final3}
		\| u\|_{H^{m-1}(\Omega')}\leq C\Big(\|f\|_{C^{2m-2}(\Omega)}^2+\|u\|_{L^2(\Omega')}+\|u\|^2_{L^\infty(\R^N\setminus \Omega')} \Big)^{\frac12}.
	\end{equation}
	Fix $\Omega'\subset\subset\Omega$ and let $\Omega_i\subset\subset\Omega$, $i=1,\ldots,7$ and $\eta\in C^{\infty}_c(\Omega_7)$ as in the proof of Corollary \ref{thm:differentiability}. Put $v=\partial^{\beta}(\eta u)$ for some $\beta\in \N_0^N$, $|\beta|=m-1$. Then $Iv=\partial^{\beta}f+\lambda v+\partial^{\beta}g_{\eta,u}$ in $\Omega_5$ by Lemma \ref{localization} and direct computation using the assumptions on $J$. From here, proceeding as in the proof of Corollary \ref{thm:differentiability} by applying Theorem \ref{thm:differentiability-step1} the claim follows.
\end{proof}

\begin{proof}[Proof of Theorem \ref{eigenfunctionsj}]
By assumption, it follows from \cite{JW20} that $\cD^j(\Omega)$ is compactly embedded into $L^2(\Omega)$. This gives the existence of the sequence of eigenfunctions and corresponding eigenvalues. The fact that the first eigenfunction can be chosen to be positive follows from the fact that $b_j(|u|)\leq b_j(u)$, Proposition \ref{wmp} and Proposition \ref{smp} (see also \cite{JW19}). Now statement (1) follows from Theorem \ref{thm-bdd-full-domain} (with $h=f=0$) and statement (2) follows directly from Corollary \ref{thm:differentiability3}.
\end{proof}

\begin{proof}[Proof of Theorem \ref{poissonj}]
The first part follows from the Poincar\'e inequality, i.e. under the assumptions it holds $\Lambda_1(\Omega)>0$, and Theorem \ref{thm-bdd-full-domain} with $h=0=\lambda$. The last assertion follows from Corollary \ref{thm:differentiability3}.
\end{proof}

\appendix

\section{An inequality}

The following is a variant of \cite[Lemma 10]{DK11} (see also \cite[Lemma 5.1]{JW20}).

\begin{lemma}
	\label{iterationlemma}
	Let $q\in L^1(\R^N)$ be a nonnegative even function with $q=0$ on $\R^N\setminus B_r(0)$ for some $r>0$. Let $\Omega\subset \R^N$ open and $x_0\in \Omega$ such that $B_{2r}(x_0)\subset \Omega$. Then for all measurable functions $u:\Omega\to \R$ we have
	\[
	b_{q\ast q,B_r(x_0)}(u)\leq 4\|q\|_{L^1(\R^N)} b_{q,\Omega}(u).
	\]
\end{lemma}
\begin{proof}
	In the following, we identify $u$ with its trivial extension $\tilde{u}:\R^N\to\R$, $\tilde{u}(x)=u(x)$ for $x\in \Omega$ and $\tilde{u}(x)=0$ otherwise. Denote $g(x,y)=(u(x)-u(y))^2$ for $x,y\in \R^N$. Note that we have
	\[
	0\leq g(x,y)=g(y,x)\leq 2g(x,z)+2g(y,z)\qquad\text{ for all $x,y,z\in \R^N$.}
	\]
	By Fubini's theorem we have
	\begin{align*}
	\int_{B_r(x_0)}&\int_{B_r(x_0)}g(x,y) (q\ast q)(x-y)\ dxdy=	\int_{B_r(x_0)}\int_{B_r(x_0)} \int_{\R^N} g(x,y) q(x-z)q(y-z)\ dzdxdy\\
	&\leq	2\int_{B_r(x_0)}\int_{B_r(x_0)} \int_{\R^N} [g(x,z)+g(y,z)] q(x-z)q(y-z)\ dzdxdy\\
	&\leq  4\int_{B_r(x_0)}\int_{\R^N}g(x,z)q(x-z) \int_{\R^{N}}q(y-z)\ dydzdx= 4\|q\|_{L^1(\R^N)}\int_{B_r(x_0)}\int_{\R^N}g(x,z)q(x-z)\ dzdx.
	\end{align*}
	Note that since $q=0$ on $\R^N\setminus B_r(0)$, $q$ is even, and $B_r(x)\subset B_{2r}(x_0)\subset \Omega$ for any $x\in B_r(x_0)$, we have
	\begin{align*}
	\int_{B_r(x_0)}\int_{\R^N}g(x,z)q(x-z)\ dzdx&=\int_{B_r(x_0)}\int_{B_r(x)}(u(x)-u(z))^2q(x-z)\ dzdx\leq 2b_{q,\Omega}(u).
	\end{align*}
\end{proof}

\bibliographystyle{amsplain}

\end{document}